\documentclass{siamltex}
\usepackage[latin1]{inputenc}
\usepackage[english]{babel}
\usepackage{graphicx}
\usepackage{amsmath,amsfonts,amssymb}
\usepackage{bbm,wasysym,stmaryrd}
\usepackage{subfigure}
\usepackage{color}
\usepackage{hyperref}
\usepackage[normalem]{ulem}

\usepackage{color}
\usepackage{soul}
\usepackage{graphicx,subfigure}
\usepackage{cancel}

\newcommand\N{{\mathbb N}}
\newcommand\R{{\mathbb R}}

\def\CC{{\mathcal C}}

\def\JJ{{\mathcal J}}

\def\LL{{\mathcal L}}

\definecolor{lightblue}{rgb}{0.30,0.30,1}	% light blue

\renewcommand{\L}{\mathbb{L}}
\renewcommand{\P}{\mathbb{P}}

\newcommand{\eps}{{\varepsilon}}

\newtheorem{theo}{Theorem}
\newtheorem{prop}[theo]{Proposition}
\newtheorem{lem}[theo]{Lemma}

\newtheorem{rem}[theo]{Remark}

\renewcommand{\theequation}{\thesection.\arabic{equation}}

\newcommand{\cris}[1]{\textcolor{blue}{#1}}
\newcommand{\john}[1]{{\color{red}{#1}}}
\newcommand{\johnNew}[1]{{\color{cyan}{#1}}}
\newcommand{\revision}[1]{#1}

\renewcommand{\cris}[1]{\revision{#1}}
\renewcommand{\john}[1]{\revision{#1}}
\renewcommand{\johnNew}[1]{\revision{#1}}
\renewcommand{\sout}[1]{}

\title{Clamping and Synchronization in the strongly coupled FitzHugh-Nagumo model}

\author{Cristobal Qui\~{n}inao\footnotemark[1] \and Jonathan D. Touboul\footnotemark[3]}
\renewcommand{\thefootnote}{\fnsymbol{footnote}}
\footnotetext[1]{Universidad de O'Higgins, Instituto de Ciencias de la Ingenier\'ia, Avda. Libertador Bernardo O'Higgins 611, Rancagua, Chile, cristobal.quininao@uoh.cl. Fondecyt/Conicyt Postdoctorado Grant 3170435}
\footnotetext[3]{Department of Mathematics and Volen National Center for Complex Systems, Brandeis University, 415 South Street, Waltham, MA, USA, jtouboul@brandeis.edu}

\begin{document}
\maketitle

\renewcommand{\thefootnote}{\arabic{footnote}}
\begin{abstract}
We investigate the dynamics of a limit of interacting FitzHugh-Nagumo neurons in the regime of large interaction coefficients. We consider the dynamics described by a mean-field model given by a nonlinear evolution partial differential equation representing the probability distribution of one given neuron in a large network. The case of weak connectivity previously studied displays a unique, exponentially stable, stationary solution. Here, we consider the case of strong connectivities, and exhibit the presence of possibly non-unique stationary behaviors or non-stationary behaviors. To this end, using Hopf-Cole transformation, we demonstrate that the solutions exponentially concentrate around a singular Dirac measure as the connectivity parameter diverges, centered at the zeros of a time-dependent continuous function. We next characterize the points at which this measure concentrates, and exhibit a particular solution corresponding to a \john{Dirac measure concentrated on a time-dependent point satisfying} an ordinary differential equation identical to the original FitzHugh-Nagumo system. This solution may thus feature multiple stable fixed points or periodic orbits, respectively corresponding to a clumping of the whole system at rest, or a synchronization of cells on a periodic solution. We illustrate these results with numerical simulations of neural networks with a relatively modest number of neurons and finite coupling strength, and show that away from the bifurcations of the limit system, the asymptotic equation recovers the main properties of more realistic networks. 

\end{abstract}

\newpage

\section{Introduction}
The study of the large-scale dynamics of neural networks is currently a great endeavor in computational neuroscience. In line with these studies, we investigate here the dynamics of a large-scale network coupled with electrical synapses, in the regime where the conductance of those synapses (coupling coefficient) is large. Electrical connections between neurons are one of the two main modalities by which neurons communicate. Contrasting with the widespread chemical synapses that transmit large, stereotyped action potentials (or spikes) fired when the neuron's electrical potential is sufficiently depolarized, electrical synapses (or gap-junctions) communicate smaller variations of neuron's voltage by direct ionic exchanges through the neuronal membranes. 

\cris{From the neuroscience point of view, the question of the dynamics of large networks with strong electrical coupling is relevant to test a hypothesis proposed in the biological and computational neuroscience literature, according to which an enhanced transmission of currents through gap junctions could support the emergence of synchronized activity~\cite{bennett2004electrical,hjorth2009gap}. Oscillations constitute a highly significant pattern of brain activity, reportedly related to various cognitive processes including as memory, attention and sleep~\cite{wang2010neurophysiological}, and abnormal synchrony is observed in various pathologies, notably epilepsy and Parkinson's disease~\cite{buzsaki2004neuronal}. Despite this relevance in applications, the theoretical understanding of noisy networks of excitable cells with gap junctions is still limited. In that domain, remarkable works have addressed weak noise and/or weak coupling~\cite{pfeuty2005combined} regimes, or of non-excitable neuron models~\cite{ostojic2009synchronization}. A notable exception is the very recent preprint of Lu\c{c}on and Poquet~\cite{lucon2018} studying, using methods from multiple timescales dynamics and hyperbolic invariant manifolds, oscillatory behaviors in related system. In all of these studies the regime of large coupling has been largely overlooked. The present study will introduce methods of the analysis of Partial Differential Equations (PDEs)  to the domain of computational neurosciences, allowing characterizing the dynamics of large networks in the large coupling limit, and the very specific dynamics associated. }

We focus on a simplified model of spiking neural network with electrical coupling, the FitzHugh-Nagumo (FhN)~\cite{fitzhugh1955,nagumo1962active} system. This central model in computational neuroscience captures, in a simplified framework, many important properties of more complex models, particularly of the canonical Hodgkin-Huxley equation~\cite{hodgkin-huxley:52}. The FhN model describes the dynamics of the voltage variable $v$ of a cell coupled to a recovery variable $x$ accounting for a variety of outwards currents
through the equations:
\[\begin{cases}
\dot v = -v(v-\lambda)(v-1) - x + I\\
\dot x = -a x + bv,
\end{cases}\]
where $\lambda$ controls the level of excitability of the cell, $I$ accounts for inwards currents, and the parameters $a>0$ and $b\geq 0$ describe the kinetics (timescale and voltage-activation) of the recovery variable. Within a network composed of $n$ FitzHugh-Nagumo neurons with voltage and adaptation $(v^{i},x^{i})_{i=1\cdots n}$, the current received by each neuron is composed of an extrinsic part, assumed noisy, and the sum of currents received others cells in the network. In the simplest model of a fully connected network with electrical synapses, the deterministic part of the current received by neuron $i$ is given by
\[I_{\john{\sout{gap}}}=\frac {\john{J}} n \sum_{j=1}^{n} (v^{j}-v^{i})+I_{ext},\]
\john{where $J$ is the conductance of the electrical synapse and $I_{ext}$ is the deterministic part of the external input}. We thus obtain the stochastic network equation:
\cris{\begin{equation}\label{eq:network}\begin{cases}
\displaystyle{dv^{i}_{t} = \Big(-v^{i}_{t}(v^i_t-\lambda)(v^i_t-1) +I_{ext} - x^{i}_{t} + \frac {\john{J}} n \sum_{j=1}^{n} (v^{j}_t-v^{i}_t) \Big)\,dt + \sigma dW^{i}_{t}}\\
dx^{i}_{t} = \left(-a x^{i}_{t} + bv^{i}_{t}\right)\,dt,
\end{cases}\end{equation} 
where $(W^{i}_{t})$ is a collection of $n$ independent Brownian motions accounting for stochastic fluctuations of the currents. A number of studies have discussed the dynamics of this model in distinct regimes. In particular, Zaks and collaborators studied the dynamics of this system and associated stochastic resonances in depth in the small noise regime~\cite{zaks2003noise,zaks2005noise}, and the large network size limit derived in~\cite{baladron2011mean,BFT:15,mischler2016kinetic} using probabilistic or functional analysis methods. Altogether, these results established that the system satisfies the propagation of chaos property, in the sense that any finite set of neuron converge towards independent realizations of the same process, whose law \cris{$\rho$ satisfies the nonlinear McKean-Vlasov Partial Differential Equation (PDE):
\begin{equation}\label{eq:MeanField}
\partial_t \rho = \partial_x\big((ax-bv)\rho\big)+\partial_v\left(\big(N(v)+x+\john{J}(v-\cris{\JJ[\rho]}))\big)\rho\right)+\frac{\sigma^2}{2}\partial^2_{vv}\rho
\end{equation}
with $N(v)=v(v-\lambda)(v-1)+I_{ext}$}, and \john{$\mathcal{J}$ the functional acting on functions $\{f \in \mathbbm{L}^{1}(\R\times\R)\; :\;\int_{\R^{2}}\vert v\vert f(v,x)<\infty,\, f(v,x)\geq0\}$ such that} \cris{$$\JJ[f] =\JJ[f(\cdot,\cdot)]= \int_{\R^{2}} v\,f(x,v)dvdx.$$} \john{While we focus here on the particular form of nonlinearity $N$ of the FhN model, our results essentially exploit its smoothness and cubic decay at infinity, and it shall be possible to extend our results to other choices of intrinsic dynamics $N$ satisfying those constraints. Moreover, in the sequel, we assume that $\sigma=\sqrt2$; all results are valid for any $\sigma>0$.}}

In~\cite{mischler2016kinetic}, we studied well-posedness of this equation and showed that there exists a unique solution provided that initial conditions satisfy a few regularity and decay at infinity assumptions (see Theorem~\ref{th:mischler} below). Moreover, we showed that when the connectivity is weak enough ($\john{J}\approx 0$), there exists a unique stationary solution which is exponentially nonlinearly stable. Heuristically, in the low coupling limit, neurons have a behavior similar to uncoupled neurons and, as a whole, the system distributes on the associated stationary distribution. The proof explicitly relies on the properties of the uncoupled system, together with a fine analysis of the spectrum in the weakly connected regime to assess persistence of stationary solutions and their stability. In the weak coupling limit, neurons are asynchronous. In sharp contrast, the regime of large coupling that we shall study here will yield highly synchronized or clamped dynamics.

To study the role of large connectivity in the dynamics of neuronal networks, we will thus analyze here the behavior of the solutions to the nonlinear FhN system in the large $n$ limit and when the coupling coefficient $J$ is large. In that case, the nonlinear term in the PDE~\eqref{eq:MeanField} becomes prominent, and methods relying on comparisons with the uncoupled linear case are no longer efficient. Mathematically, we focus on the limit $\john{J}\rightarrow+\infty$, a grotesque regime from the biological viewpoint (since currents transmitted by gap junctions are bounded), which allows a detailed mathematical analysis of the role of increased electrical coupling in the synchronization of neurons. 

The paper is organized as follows. In section~\ref{sec:Model}, we introduce in detail the model studied, review the relevant literature, and summarize the main results of this manuscript. These results are proven in the following sections. In section~\ref{sec:full}, we study the properties of the solution to the limit equation in the general case, and \cris{establish} \emph{a priori} bounds on the probability density useful for our developments. \cris{In particular, a uniform upper-bounded and a tightness result are proved, ensuring that, for coupling sufficiently large, the distributions of voltage and adaptation concentrate on a compact set independent of the coupling strength}. In section~\ref{sec:preliminary}, \john{we study a sequence of approximations of the solutions obtained by replacing the cubic locally-Lipschitz drift by a globally Lipschitz map identical to the original drift on a compact set, and obtain for this sequence a number of regularity results using} maximum/comparison principle for elliptic equations. In section~\ref{sec:rel}, \john{we exploit the properties of the sequence of approximations to demonstrate} the main results of the paper regarding convergence and concentration of the solutions in the large coupling limit. These results do not provide explicitly the limit towards which the system converges. We thus address the identification of a limit of the system in section~\ref{sec:indentification}, and confirm the accuracy of this limit for finite-sized networks with bounded connectivity in section~\ref{sec:numerics} where numerical simulations exhibit the \cris{existence of multiple} stable solutions and periodic solutions consistent with the theoretical results and with the particular solution exhibited.

\section{The mean-field FitzHugh-Nagumo model}

\label{sec:Model}
We introduce here in detail the model studied throughout the manuscript and summarize the main results to be proved in the following sections. 

\subsection{Setting, Model and Definitions}
\john{The central equation~\eqref{eq:MeanField} analyzed in this manuscript was studied in~\cite{mischler2016kinetic}, and existence and uniqueness of solutions was proved under specific conditions on the initial condition summarized below.
\begin{theorem}[see~\cite{mischler2016kinetic}, Theorem 2.2] 
\label{th:mischler}
For any initial condition $\rho(0,\cdot,\cdot)$ a probability distribution on $(v,x)\in \R^{2}$ with bounded second moment and finite entropy, $\rho(0,\cdot,\cdot) \in \L^{1}(1+x^{2}+v^{2})\cap \L^{1}\log \L^{1} \cap \P(\R^{2})$, there exists a unique weak solution to~\eqref{eq:MeanField}, that is uniformly in time bounded in $\L^{1}(1+x^{2}+v^{2})$. If moreover $\rho(0,\cdot,\cdot)\in \L^{1}(e^{\kappa (x^{2}+v^{2})})$ for some $\kappa>0$, then the solution remains in this space for all times\footnote{\john{Here we used the following classical notations for functional spaces: $\L^{1}(\omega(x,v))$ denotes the weighted $\L^{1}$ space of functions, i.e. functions $f$ such that $\int \omega(v,x)\vert f(v,x)\vert\,dvdx<\infty$; $\L^{1}\log \L^{1}$ is the space of functions with finite entropy $\{f\in \L^{1}(\R^{2}) \text{such that} f\geq 0 \; a.e.\; \text{ and } \int_{\R^{2}} f \log(f)<\infty\}$, and $\P(\R^{2})$ is the space of probability distribution functions $\{f\in \L^{1}(\R^{2}) \text{ such that } f\geq 0 \; a.e.\; \text{ and } \int_{\R^{2}} f(v,x)=1\}$.}}.
\end{theorem}

In the present manuscript, we rely on this result and concentrate our attention on the behavior of the solutions to this equation as the coupling parameter $\john{J}$  diverges, or, as we note in the sequel, $\eps=J^{-1}$ goes to $0$. Moreover, to avoid difficulties associated with the hypo-elliptic degeneracy of the diffusion term already handled in~\cite{mischler2016kinetic} (noiseless adaptation in the network equations), we study a slightly modified version of equation~\eqref{eq:MeanField} incorporating a vanishing diffusion on the adaptation variable with diffusion coefficient $\sqrt{2\eps}$. Convergence results proved herein are not specific to this choice of diffusion, and may be easily generalized to other vanishing diffusion coefficients, and even possibly to the hypoelliptic case (absence of diffusion). However, the presence of this small diffusion allows us to concentrate on the main difficulties associated with diverging connectivity, and while the specific square-root scaling of the diffusion coefficient shall play no role in the existence of a limit, it controls the speed of convergence, and our choice was driven by the fact that a square-root decay introduces no additional scaling in the convergence (see section~\ref{sec:indentification}).}

 \cris{We emphasize that the results of theorem~\ref{th:mischler} can be easily extended to cases with a diffusion in the adaptation variable. The interested reader will notice that the presence of a diffusion only simplifies the derivations and results of existence and uniqueness of solutions apply, \emph{mutatis mutandis}, to the present case. Indeed, the proof of Theorem~\ref{th:mischler} relies on a priori bounds that, as a consequence of the lack of second $x$ derivative, use a specific norm including cross derivative contribution. In this norm, an extra second derivative term in $x$ associated with a diffusion on the adaptation variable will only add an additional negative contribution, and thus all upper-bounds derived for the degenerate case persist; one should therefore obtain as consequences existence, uniqueness and regularity of the solutions.}

Therefore, for $\eps>0$, we are concerned with the behavior of $g_\eps(t,x,v)$, solutions to the equation
\begin{multline}
\label{eq:evo1}
 \partial_t g_\eps(t,x,v)=\partial_x\big((ax-bv)g_\eps(t,x,v)+\eps\partial_x g_\eps(t,x,v)\big)
 \\+\partial_v\left(\big(N(v)+x+\eps^{-1}(v-\cris{\JJ[g_\eps]})\big)g_\eps(t,x,v)+\partial_vg_\eps(t,x,v)\right),
\end{multline}
coupled with \john{the first moment of $g_\eps$, which is} \cris{the time-dependent function}
\begin{equation}
\label{eq:It}
\cris{\JJ[g_\eps]}(t)=\cris{\JJ[g_\eps(t,\cdot,\cdot)]}=\int_{\R}\int_\R vg_\eps(t,x,v)dxdv,
\end{equation}
accounting for the network-generated average current received by the neurons. 

\john{Throughout the paper, we will consider an initial condition on the probability distribution that satisfies the following assumption, labeled (H) in the sequel:

\medskip
\textbf{(H)} The initial conditions $g_\eps(0,x,v)=g_\eps^0(x,v)$ are a sequence of regular probability measures (at least $C^3(\R^2)$) with a uniform upper bound for the derivatives, and moreover, there exist two positive constants $A$ and $B$, with $A\leq \min(a,1)$, such that:
 $$ 
 \sup_{0<\eps<1}\psi_\eps(0,x,v)=\sup_{0<\eps<1}\eps\log g_\eps^0(x,v)\leq-\frac{A}2(v^2+x^2)+B.
 $$
 
\begin{rem}
 Notice that hypothesis (H) is consistent with the hypotheses of Theorem~\ref{th:mischler}. An initial condition $g_\eps^0$ satisfying that assumption is indeed in $\P(\R^2)$ and has the upper-bound
 $$
  g_\eps^0(x,v) = e^{\psi_\eps(0,x,v)/\eps}\leq C e^{-A(x^2+v^2)/2\eps},
 $$
ensuring that for each $\eps$, we have that $g_\eps^0(x,v)\in \L^{1}(1+x^{2}+v^{2})\cap \L^{1}\log \L^{1} \cap \P(\R^{2})$. Moreover, note that the initial condition is infinitely differentiable in the weak sense, owing to its integrability in a weighted $\L^1$ space with exponential weights. 
\end{rem}

}

Equation~\eqref{eq:evo1} is in divergence form, implying in particular that the positivity principle holds true and that the mass is conserved; for an initial condition given by a \john{probability distribution (a }\cris{nonnegative solution with unit mass}\john{)}, $g_\eps$ thus remains for all times a well-defined \john{probability distribution.} The limit $\eps\rightarrow0$ corresponds to a strong connectivity regime in a FitzHugh-Nagumo simplified equation. To understand formally the behavior of the family $\{g_\eps\}_\eps$, we rewrite equation~\eqref{eq:evo1} as
\begin{multline}
\nonumber
 \frac{\partial_tg_\eps(t,x,v)}{g_\eps(t,x,v)}=\left( a+N'(v)+\eps^{-1}\right)+(ax-bv)\frac{\partial_xg_\eps(t,x,v)}{g_\eps(t,x,v)}
 \\
 +\left(  N(v)+x+\eps^{-1}(v-\cris{\JJ[g_\eps]})\right)\frac{\partial_vg_\eps(t,x,v)}{g_\eps(t,x,v)}  
 +\eps\frac{\partial^2_{xx} g_\eps(t,x,v)}{g_\eps(t,x,v)}
 +\frac{\partial^2_{vv} g_\eps(t,x,v)}{g_\eps(t,x,v)}.
\end{multline}
Notice that for $\eps$ small, the time derivative on the left hand side becomes negligible compared to the terms on the order of $\eps^{-1}$. To rigorously handle this divergence, we consider the Hopf-Cole transformation $\psi_\eps = \eps \log g_\eps$. The map $\psi_{\eps}$ satisfies the equation:
\begin{multline}
\label{eq:hct}
\partial_t\psi_\eps
=
\left(\eps a+\eps N'(v)+1\right)
+
(ax-bv)\partial_x\psi_\eps
\\
+
\left( N(v)+x+\eps^{-1}(v-\cris{\JJ[g_\eps]})\right)\partial_v\psi_\eps
\\
+
|\partial_x\psi_\eps|^2+\eps\partial^2_{xx}\psi_\eps
+
\eps^{-1}|\partial_v\psi_\eps|^2+\partial^2_{vv}\psi_\eps,
\end{multline}
and identifying the terms with the same order of magnitude shall characterize the limit solution $\psi=\lim_{\eps\to 0} \psi_{\eps}$. This limit has several important properties: first, it is non-positive ($\psi(t,x,v)\leq 0$ for all $(t,x,v)$) \cris{and we will show that, as a consequence of the mass conservation property, it reaches $0$. Moreover,} the support of $g_{\eps}$ is given by the zeros of $\psi$ when $\eps\to 0$. The paper focuses on characterizing the limit $\psi$. This problem raises challenging questions, especially due to the two-dimensional nature of the equation and to the presence of multiple solutions. Indeed, considering the diverging terms of order $\eps^{-1}$ in equation~\eqref{eq:hct} elucidates the possible dependence of $\psi$ in $v$:
\begin{equation}
\label{eq:limit1}
 0=(v-\cris{\alpha(t)})\partial_v\psi(t,x,v)+|\partial_v\psi(t,x,v)|^2,
\end{equation}
\cris{for $\alpha(t)$ the limit of $\JJ[g_\eps(t\cdot,\cdot)]$ in a sense to be \john{specified below}.} Solutions of this equation include functions independent of $v$, as well as a non-trivial quadratic solution\footnote{Continuous combinations of solutions independent of $v$ and the quadratic map are also solutions of the equation in the weak sense.} (see section~\ref{sec:indentification}):
\[\psi(t,x,v)=-\frac12(v-\cris{\alpha(t)})^2+\phi(x,t),\]
where $\cris{\alpha(t)}$ is the first coordinate of the solution of the FitzHugh -Nagumo equation:
\cris{\begin{equation}\label{eq:FhNLim}
\begin{cases}
\frac{d\alpha}{dt} & = -N(\alpha(t)) -\beta(t)\\
\frac{d\beta}{dt} & = a \beta(t) -b \alpha(t)\\
\end{cases}
\end{equation}
with initial condition given by the average value of $v$ and $x$ of the initial distribution. For that solution, the condition $\psi(t,x,v) \leq 0$ implies that $\phi(x,t)\leq 0$, and we can express explicitly one particular solution by considering the equation associated with the terms of order $1$ in equation~\eqref{eq:hct}:
\begin{equation}\label{eq:psiSol}
\psi(t,x,v)=-\frac12(v-\alpha(t))^2-\frac a 2 (x-\beta(t))^2
\end{equation}
where $\beta$ is actually the second coordinate of equation~\eqref{eq:FhNLim}. Extensive numerical simulations provided in section~\ref{sec:numerics} confirm that the network equation converges towards this particular solution of the system. }

\cris{As mentioned above, the set where $\psi$ reaches zero is particularly relevant: indeed, for $\eps$ small, the support of the distributions $g_\eps(t,\cdot,\cdot)$ concentrate exponentially fast on this set. This observation implies that the support of the solution in the voltage variable for that particular solution shrinks to a \john{single point} $v=\alpha(t)$ and $x=\beta(t)$}, allowing to show that the system may be clamped at a given point when the solution of the equation~\eqref{eq:FhNLim} has a stable fixed point, that multiple clamped stationary solutions exist, or even synchronized periodic solutions, depending on the nonlinearity $N$ and the parameters $(a,b)$.

Following the ideas of~\cite{barles2009concentration,leman2014influence}, we shall prove most of these results rigorously. In the present case, the \john{fact that} the drift $N(v)$ \john{is not globally Lipschitz-continuous}, the nonlinearity of the problem and the presence of the diverging interaction term $\eps^{-1}$ leads to difficulties that need careful consideration.

Because we expect a concentration of the distribution on a compact set as the coupling increases, the non-global Lipschitz continuity of the drift will not be a critical aspect in the limit $\eps\to 0$. Rigorously, this concentration will allow us to describe the system in the small $\eps$ limit through a simpler model in which (i) $N(v)$ is replaced by a smooth truncated function $N_M(v)$ identical to $N(v)$ on a compact interval and having a ``linear growth'' at infinity, and (ii) a linear interaction term. In detail, for $M>0$ fixed, we consider the continuously differentiable \emph{truncated driving function}:
$$
N_M(v)=\begin{cases}
N(-M)+N'(-M)(v+M), & v<-M,
\\
N(v), & -M\leq v\leq M,
\\
N(M)+N'(M)(v-M), & v>M,
\end{cases}
$$
\cris{which is globally Lipschitz-continuous. Now, each function $g_\eps$ will be approximated by a sequence of functions $f_{\eps}^{M}$ satisfying the equation:
\begin{multline}
\label{eq:truncated}
 \partial_t f_\eps^M(t,x,v)=\big(a+N_M'(v)+\eps^{-1}\big)f_\eps^M(t,x,v)+(ax-bv)\partial_xf_\eps^M(t,x,v)
 \\
 +\big(N_M(v)+x+\eps^{-1}(v-\cris{\JJ[g_\eps]})\big)\partial_vf_\eps^M(t,x,v)
 \\
 +
  \eps\partial^2_{xx} f_\eps^M(t,x,v)
 +\partial^2_{vv} f_\eps^M(t,x,v).
\end{multline}
where $\cris{\JJ[g_\eps]}$ is no longer implicitly defined (contrasting with equation~\eqref{eq:evo1}), but is the first moment in $v$ of $g_\eps$, which will be shown to exist and satisfy proper regularity conditions. For this new equation $\JJ[g_\eps]$ can be interpreted as an \emph{external current}, so that the sequence of approximations constructed are solutions of a \emph{linear} equation~\eqref{eq:truncated}, and we will exploit this linearity to characterize a number of useful properties for our purposes}. \johnNew{Moreover, since the last equation is linear, has globally Lipschitz coefficients, and has a non-degenerate diffusion term, classical results ensure weak existence and uniqueness of solutions, in particular when the initial conditions satisfy assumption (H), in the same sense as in Theorem~\ref{th:mischler}.}

\subsection{Summary of the main results}
Now that the setting and equations have been posed, we summarize below the main mathematical results of this paper. The first result characterizes in detail the solutions to the truncated mean-field FitzHugh-Nagumo system and the convergence result for large coupling. 

\begin{theo}
\label{theo:main}
\john{Let $f^M_\eps(t,x,v)$ be the solution to the truncated equation~\eqref{eq:truncated} with initial conditions $f_\eps^{M,0}$ satisfying assumption (H) uniformly in $M$, i.e., such that there exists $A>0$ and $B>0$, independent of $M$, with $A\leq \min (a,1)$, such that
  $$
\sup_{0<\eps<1}\eps\log f_\eps^{M,0}(x,v)\leq -\frac{A}2(v^2+x^2)+B.
  $$
Then the family of functions $\varphi^M_\eps(t,x,v)=\eps\log f^M_\eps(t,x,v)$ are well defined. Moreover,
  \begin{enumerate} 
	  \item for each $M$, the sequence of functions $(\varphi^{M}_\eps)_{\eps\in (0,1)}$ \johnNew{is relatively compact, thus} converges locally uniformly \johnNew{on subsequences, as $\eps\to 0$,} to a function $\varphi^M\in C\big((0,+\infty)\times\R^2\big)$, a viscosity solution to the following equation:
 \begin{equation}
\label{eq:viscosity}
 0=(v-\john{\alpha}(t))\partial_v\varphi^M+|\partial_v\varphi^M|^2,\qquad \max_{(x,v)\in\R^2}\varphi^M(t,x,v)=0,
\end{equation}
with $\john{\alpha}(t)$ \johnNew{an adherent point of the sequence $(\JJ[g_\eps] (t))_{\eps\in (0,1)}$}. 
\item Denoting by $f^M$ the weak limit of $f^{M}_\eps$ as $\eps$ vanishes, we have that a.e. in $t$, $\supp f^M(t,\cdot,\cdot)\subset\{\varphi^M(t,\cdot,\cdot)=0\}.$ 
\end{enumerate}}
\end{theo}

\medskip
Our second main result shows that the properties proved for the truncated problem generalize to the original problem.\\

\begin{theo}
\label{theo:main2}
 \john{Let $g_\eps(t,x,v)$ be the solution of~\eqref{eq:evo1}-\eqref{eq:It} with initial conditions $g^0_\eps(x,v)$ satisfying assumption (H), and $f_\eps^{M}(t,x,v)$ the solution of the truncated problem~\eqref{eq:truncated} with the same initial condition: $f_\eps^{M,0}(x,v)=g^0_\eps(x,v)$. Then, there exists $M_0$ large enough such that for any $M>M_0$ and every regular test function $\phi\in C_c^\infty(\R^2)$ it holds that:
 $$
\Big|\int_{\R^2} \phi(x,v) f_\eps^{M}(t,x,v)\,dxdv-\int_{\R^2} \phi(x,v)g_\eps(t,x,v)\,dxdv\Big|\xrightarrow{\eps\rightarrow0}0.
 $$
 In particular, $g_\eps$ is converging towards $f^M$, the weak limit of $f_\eps^{M}$ as $\eps$ is going to 0.}
\end{theo}

\johnNew{The above results provide a convergence in sense that for any subsequence of $g_\eps$, we can extract a subsequence that converges. A full convergence result follows if there exists a unique possible limit. However, proving in a general case uniqueness of the limit for this type of equations is a notoriously complex problem. Previous works have either left this question open~\cite{barles2007concentrations}, stated results valid upon extraction of a subsequence as in~\cite{barles2008dirac}, or elegantly showed uniqueness of solutions under additional assumptions on the structure on the limit: assuming the limit is a Dirac measure at a single point in~\cite{barles2008dirac} (\emph{monormorphic} populations in the realm of population biology), or in~\cite{mirrahimi:12} a combination of two Dirac measures (\emph{dimorphic} populations). The same difficulty arises here, as mentioned above. However, in the present case, the coupling term suggests, as in monomorphic populations, that the voltage and adaptation variables converge to a Dirac mass. Indeed, the coupling term  strongly constrains the dispersion of trajectories, and any solution with a voltage away from the average voltage $\JJ[g_\eps](t)$ will be quickly attracted to $\JJ[g_\eps](t)$. To appreciate this property, it is useful to consider the McKean-Vlasov equation governing the stochastic trajectories associated with~\eqref{eq:evo1}\footnote{We will see that $\JJ[g_\eps](t)$ is well-defined for all times regardless of convergence results in $\eps$, and can thus be considered as an external input for the above equation.}:
\[\begin{cases}
dv_t = \Big(-N(v_t)-x_t +\frac{1}{\eps}(\JJ[g_\eps](t)-x_t)\Big)\,dt + \sqrt{2}dW_t\\
dx_t = (b v_t-ax_t)\,dt,
\end{cases}
\]
A neuron with initial condition $(v_0,x_0)$ quickly compensates any deviation of $v_0$ from $\JJ[g_\eps](0)$ within a time of order $\eps$. Indeed, at this timescale, the variables $(\tilde{v}_t,\tilde{x}_t) = (v_{\eps t},x_{\eps t})$ are given by the stochastic equations: 
\[\begin{cases}
d\tilde{v}_t = \eps \Big(-N(\tilde{v}_t)-\tilde{x}_t\Big)\,dt +(\JJ[g_\eps](\eps t)-\tilde{x}_t)\,dt + \sqrt{2\eps}dW_t\\
d\tilde{x}_t = \eps(b \tilde{v}_t-a\tilde{x}_t)\,dt.
\end{cases}
\]
For small $\eps$, $\tilde{x}_t$ is thus almost constant, and $\tilde{v}_t$ has a dynamics a stochastic particle in a perturbed quadratic potential (with a perturbation of order $\eps$), in the regime of small noise. Within a time of order $\eps$, the particle converges exponentially fast towards $\JJ[g_\eps](0)$, and deviations from that value are extremely rare, and their probability may be evaluated using Freidlin-Wentzell type of estimates~\cite{FredilinWentzell}. This formal argument is not only valid at the initial time. If at a given time $t>0$ a trajectory deviates from $\JJ[g_\eps](t)$, the same argument allows us to prove that $(\tilde{v}_s,\tilde{x}_s)=({v_{t+\eps s}},{x_{t+\eps s}})$ quickly returns to $\JJ[g_\eps](t)$, and therefore the particular coupling of the FitzHugh-Nagumo equation formally suggests to consider only Dirac-distributed measures when $\eps\to 0$. And as indicated above (eq.~\eqref{eq:psiSol}), there exists a single Dirac measure satisfying the limit equation, and this Dirac mass is centered at a point $(\alpha(t),\beta(t))$ satisfying the FitzHugh-Nagumo equation~\eqref{eq:FhNLim}, in turn formally ensuring convergence of the full sequence $g_\eps$ towards this Dirac distribution.
}

\smallskip
\john{Note that, because of the concentration of the solutions, the truncation is no longer active when $M$ is large enough. In particular, we do not have a double-limit problem and associated issues, as the above convergence is valid for fixed $M$ large enough.}
The proofs of these theorems use concentration techniques developed initially in the context of evolutionary and adaptive systems (see e.g.~\cite{diekmann2005dynamics,champagnat2011evolutionary,barles2007concentrations,barles2008dirac} for Lotka-Volterra parabolic and integral equations). Moreover, the notion of viscosity solution convergence and characterization of a limit in terms of a Hamilton-Jacobi equation was well documented for those systems (for a general introduction to the theory see e.g~\cite{crandall1992user,barles1994solutions,evans1997partial}\footnote{\john{We recall that a viscosity solution for a partial differential equation, roughly speaking, the limit of a sequence of solutions for an associated regularized problem as the regularization vanishes. Here, our regularization is both having a finite the coupling term and having a small diffusion. Typical examples of viscosity solutions arise are degenerate diffusions problems, regularized by adding a small diffusion term, and viscosity solution are the limits of the sequence of solutions obtained as the diffusion coefficient vanishes. }}).

To prove these results, we start in section~\ref{sec:full} by a thorough study of the properties of the solution to the general case of equation~\eqref{eq:evo1}, with a particular focus on \emph{a priori} bounds and tightness. We then study in section~\ref{sec:preliminary} the truncated equation, specifically regularity estimates obtained from maximum/comparison principles for elliptic equations. These technical steps will provide us with the main elements used to prove, in section~\ref{sec:rel}, Theorems~\ref{theo:main} and~\ref{theo:main2}.

\section{Uniform upper bounds for the general non-local equation}
\label{sec:full}
We start by studying the general problem~\eqref{eq:evo1} with no truncation, and show some useful \emph{a priori} uniform bounds, particularly that $ \cris{\JJ[g_\eps]} (t)$ and its derivatives are uniformly controlled and converging locally uniformly towards some continuous function $\john{\alpha}(t)$. Moreover, by taking initial conditions adequately, we will control, for each $\eps$, the behavior of $g_\eps$ and their Hopf-Cole transformations when time increases. Both results will be crucial for our proof of Theorem~\ref{theo:main2} to show that the difference between $f^{M}_\eps$ for $M$ larger than a constant and $g_\eps$ is arbitrarily small as $\eps$ vanishes.
\cris{
\begin{lem}
\label{lem:Iboundfull}
Consider the solution of equation~\eqref{eq:evo1} with initial condition $g_\eps^0$ satisfying assumption (H). 
There exists a positive constant $C_I$ independent of $\eps<1$ such that
$$
 \begin{cases}
 |  \cris{\JJ[g_\eps]} (t) |\leq C_I & \text{ for all }t\geq0,\\
 \Big|\frac{d}{dt} \cris{\JJ[g_\eps]} (t)\Big|\leq C_{I} & \text{ for all }t\geq0,
 \end{cases}
$$
in a weak sense. Consequently, we have locally uniform convergence along subsequences of $ \cris{\JJ[g_\eps]} (t)$, as $\eps$ goes to $0$.
\end{lem}
}
\begin{proof} 
\cris{
\john{This result is proved by controlling order-four moments of the solutions of equation~\eqref{eq:evo1} and using their natural relationship with the quantities we characterize in this lemma.} It follows from hypothesis (H) that
$$
 \int_{\R^2}(x^4+v^4)g_\eps^0(x,v) \leq \int_{\R^2}(x^4+v^4)e^{-A(x^2+v^2)/2\eps}\leq \int_{\R^2}(x^4+v^4)e^{-A(x^2+v^2)/2}=\kappa_0,
$$
with $\kappa_0$ independent of $\eps<1$. The rest of the proof is direct consequence of equation~\eqref{eq:evo1}.} Indeed, we have that
\begin{multline*}
  \frac1{4}\frac{d}{dt}\int_{\R^2} (x^4+v^4)g_\eps \john{(t,x,v)\,dx dv}
  = \frac14\int_{\R^2} (x^4+v^4)\left[\eps\partial^2_{xx}g_\eps+\partial^2_{vv}g_\eps\right]
  \\
  \quad+\frac14\int_{\R^2} (x^4+v^4)\partial_x\big[\big(ax-bv\big)g_\eps\big]
 +\frac14\int_{\R^2} (x^4+v^4)\partial_v\big[\big(N(v)+x+\eps^{-1}(v- \cris{\JJ[g_\eps]} (t))\big)g_\eps\big]
  \\
  = \int_{\R^2} \Big[x^2\big(3\eps-ax^2+bvx\big)+v^2\big(3-vN(v)-vx\big)-\eps^{-1}v^3(v- \cris{\JJ[g_\eps]} (t))\Big]g_\eps.
 \end{multline*}
 \john{(to alleviate the notations, we only recalled the variables $(t,x,v)$ in the notation of the map $g_\eps$ and the integration variables in the lefthand side). This expression involves only polynomial terms and a more complex non-local term multiplied by a coefficient $\eps^{-1}$. We start by proving that the nonlocal term is non-positive, using H\"older's inequality several times.}
Using the mass conservation property and non-negativity of $g_{\eps}$, we obtain:
\begin{multline*}
\int_{\R^2} v^3( \cris{\JJ[g_\eps]} (t)-v) g_\eps = \int_{\R^2} v^3g_\eps\int_{\R^2} vg_\eps-\int_{\R^2} v^4 g_\eps
\\
\leq\int_{\R^2} |v|^3g_\eps\left(\int_{\R^2} v^4g_\eps\right)^{1/4}-\int_{\R^2} v^4 g_\eps.
\end{multline*}
Moreover, we also have that
$$
 \int_{\R^2}|v|^3g_\eps =  \int_{\R^2}|v|^3g_\eps^{3/4}g_\eps^{1/4}\leq\left( \int_{\R^2} v^4 g_\eps\right)^{3/4}\left( \int_{\R^2} g_\eps\right)^{1/4}=\left( \int_{\R^2}v^4 g_\eps\right)^{3/4},
$$
thus, for any $\eps$ it holds that
\begin{equation}\label{eq:remove_eps_1}
-\int_{\R^2} v^3\eps^{-1}(v- \cris{\JJ[g_\eps]} (t))\big)g_\eps\leq0.
\end{equation}

\john{We can thus upperbound our estimate by the following expression, use the fact that } $N$ is a cubic polynomial with leading term $v^3$, to obtain that
\begin{align*}
\frac{d}{dt}\int_{\R^2} (x^4+v^4)g_\eps & \leq 4\int_{\R^2}\Big[3x^2\eps-ax^4+bx^3v+3v^2-v^3N(v)-xv^3\Big]g_\eps(t,v)
\\
& \leq C- c\int_{\R^2} (x^4+v^4)g_\eps,
\end{align*}
for some positive constants $c,C$ depending on all the parameters of the system but not on $\eps$ for $\eps<1$. Gronwall's lemma thus implies that:
$$
 \int_{\R^2} (x^4+v^4)g_\eps(t,x,v)\leq \max\left\{C,\int_{\R^2} (x^4+v^4)g_\eps^0(x,v)\right\},
$$
by taking $C$ larger if necessary. Finally, using once again the mass conservation principle we get that for $k\in\{1,2,3\}$ it holds that
\begin{multline*}
 \int_{\R^2}\big[|x|^k+ |v|^k\big]g_\eps(t,x,v)
 \leq
 \int_{\R^2} (1+x^4+v^4)g_\eps(t,x,v)
 \\
 \leq 1+\max\left\{C,\int_{\R^2} (x^4+ v^4\john{)} g_\eps^0(x,v)dv\right\}=:C_I,
\end{multline*}
and thus both $| \cris{\JJ[g_\eps]} (t)|$ and $\Big|\frac{d}{dt} \cris{\JJ[g_\eps]} (t)\Big|$ are uniformly bounded.

\john{These two bounds ensure that $\JJ[g_\eps]$ is a sequence of bounded and equicontinuous functions, and therefore Arzel\`a-Ascoli theorem ensures that for any $T>0$, there exists a subsequence that converges uniformly on $[0,T]$, thus implying the last statement of the lemma.}
\end{proof}
\bigskip

This uniform control on $ \cris{\JJ[g_\eps]} (t)$ and its derivative \cris{imply} that the collection of maps $\psi_\eps=\eps\log g_\eps$ is uniformly upper-bounded, and by consequence if $\eps$ is small enough, it suffices to study the limit behavior of $g_\eps$ only in a compact subset.  To demonstrate this point, we will use the following operator acting on maps continuously differentiable with respect to time and twice continuously differentiable with respect to space $C^1([0,\infty);C^2(\R^2))$:
\begin{multline*}
\LL_\eps\phi := \partial_t\phi- \left(\eps a+\eps N'(v)+1\right)
-
(ax-bv)\partial_x\phi
\\
-
\left( N(v)+x+\eps^{-1}(v- \cris{\JJ[g_\eps]} (t))\right)\partial_v\phi
\\
-
|\partial_x\phi|^2-\eps\partial^2_{xx}\phi
-
\eps^{-1}|\partial_v\phi|^2-\partial^2_{vv}\phi.
\end{multline*}
\begin{lem}
\label{lem:bound2psi}
Assume that initial conditions $\psi_\eps(0,x,v)=\eps\log g_\eps^0(x,v)$ \john{satisfy assumption (H).}
Then, there is some constants $0<A'<\min(a,1)$, $B'>0$ and $D>0$ such that, for any $\eps>0$ and any $t>0$,
\begin{equation}
\label{eq:boundphi}
\psi_\eps(t,x,v)\leq -\frac {A'} 2\big(v- \cris{\JJ[g_\eps]} (t)\big)^2-\frac {{A'} x^2}2+ B'+ Dt.
\end{equation}
\end{lem}

\begin{proof} The proof follows the classical technique of exhibiting a suitable super-solution to~\eqref{eq:hct}. Here, we consider the map 
$$\phi_\eps(t,x,v)=-\frac {A'} 2\big(v- \cris{\JJ[g_\eps](t)} \big)^2-\frac {A'x^2}2+B'+Dt$$
\john{for some constants $A'$, $B'$ and $D$ to be specified.}

\john{Assumption (H) implies the existence of some $A'<\min(1,a)$ and $B'$ such that for any $\eps\in(0,1)$, the initial condition $\psi_\eps(0,x,v)=\eps\log g_\eps^0(x,v)$ is upper-bounded by $\phi_\eps$. Indeed, assumption (H) ensures that
$$
  \sup_{0<\eps<1} \psi_\eps(0,x,v)\leq-\frac A2(v^2+x^2)+B,
 $$
 with $A\leq \min(a,1)$. Under (H), Lemma~\ref{lem:Iboundfull} implies that $ \cris{\JJ[g_\eps]} (t)$ and its time-derivative are uniformly bounded by a constant $C_I$, so that for all $\eps\in(0,1)$, we can rewrite assumption (H) as:
 \begin{align*}
 \psi_\eps(0,x,v) &\leq-\frac A2(v- \cris{\JJ[g_\eps]} (0))^2-Av \cris{\JJ[g_\eps]} (0)+\frac A2C_I^2-\frac A2x^2+B\\
 &\leq -\frac{A}{4}(v- \cris{\JJ[g_\eps]} (0))^2 -\frac A2x^2 + \frac A2C_I^2+B
 \end{align*}
implying the existence of constants $A'$ (smaller than $A$) and $B'$ (larger than $B$) for which 
$$
  \sup_{0<\eps<1} \psi_\eps(0,x,v)\leq \phi_\eps(0,x,v).
 $$
This proves that the inequality claimed in the proposition is true of the initial condition for a proper choice of constants. We now use Gronwall's lemma to show that this inequality persists in time. First of all, we note that:}
\begin{multline*}
 \LL_\eps\phi_\eps
= D
-A' \big(v- \cris{\JJ[g_\eps]} (t)\big)\frac{d \cris{\JJ[g_\eps]} (t)}{dt}
-\big(\eps a+\eps N'(v)+1\big)
\\
+A'x(ax-bv)
+A'(v- \cris{\JJ[g_\eps]} (t))\left( N(v)+x+\eps^{-1}(v- \cris{\JJ[g_\eps]} (t))\right)
\\
-A'^2x^2\eps A'
-\eps^{-1}A'^2(v- \cris{\JJ[g_\eps]} (t))^2+A'.
\end{multline*}
We start considering the terms of order $\eps^{-1}$; re-arranging those terms adequately, we get that
$$
 A'\big(v- \cris{\JJ[g_\eps]} (t)\big)^2-A'^2\big(v- \cris{\JJ[g_\eps]} (t)\big)^2=A'(1-A')\big(v- \cris{\JJ[g_\eps]} (t)\big)^2\geq0,
$$
whenever $A'\leq1$. Moreover, recalling that $N$ grows as $v^3$, that both $ \cris{\JJ[g_\eps]} (t)$ and $ \cris{\JJ[g_\eps]} (t)'$ are uniformly bounded from the result of lemma~\ref{lem:Iboundfull}, and that $A'\leq a$, it follows that
\begin{multline*}
(1+\eps)A'+D-A'\big(v- \cris{\JJ[g_\eps]} (t)\big)\frac{d \cris{\JJ[g_\eps]} (t)}{dt}-\left(\eps a+\eps N'(v)+1\right)
\\
+A'x^2(a- A')-A'b\,xv+A'\big(v- \cris{\JJ[g_\eps]} (t)\big)(N(v)+x)\geq D-C,
\end{multline*}
for some generic constant $C$ independent of $\eps\leq1$ but depending on $A'$, $a$ and $C_{I}$ (see Lemma~\ref{lem:Iboundfull}). Therefore, we can take $D$ sufficiently large such that:
$$
\LL_\eps\phi_\eps\geq D-C\geq0,
$$
Finally, the fact that $
\phi_\eps(0,x,v)=-\frac {A'}2(v- \cris{\JJ[g_\eps]} (0))^2-\frac{ A'x^2}2+B'\geq \psi_\eps(0,x,v)$ proves that $\phi_\eps$ is a super-solution of the equation $\LL_\eps\psi=0$, concluding our proof.
\end{proof}

We finish this section by stating that the same upper bound is valid for the solutions of the truncated equation~\eqref{eq:truncated}. Since the proof closely follows the arguments developed above, we do not provide it fully but outline only a few aspects that may differ. Again, it is convenient to introduce the differential operator $\LL_{\eps}^M$ acting on $C^1([0,\infty);C^2(\R^2))$ such that $\varphi_\eps^M:=\eps\log f_\eps^M$ solves the equation
\begin{multline}
\label{eq:phi}
\LL_{\eps}^M\varphi^M_\eps:=\partial_t\varphi^M_\eps
-
\left(\eps a+\eps N_M'(v)+1\right)
-
(ax-bv)\partial_x\varphi^M_\eps
\\
-
\left( N_M(v)+x+\eps^{-1}(v- \cris{\JJ[g_\eps]} (t))\right)\partial_v\varphi^M_\eps
\\
-
|\partial_x\varphi^M_\eps|^2-\eps\partial^2_{xx}\varphi^M_\eps
-
\eps^{-1}|\partial_v\varphi^M_\eps|^2-\partial^2_{vv}\varphi^M_\eps=0.
\end{multline}

\begin{lem}
\label{lem:bound2phi}
Assume that initial conditions $\varphi^M_\eps(0,x,v)=\eps\log f^M_\eps(0,x,v)$ \john{satisfy assumption (H) uniformly in $M$, i.e. for all $M$, satisfy the regularity conditions of (H) and}  are such that
$$
\varphi^M_\eps(0,x,v)\leq-\frac{A'}2\big(v- \cris{\JJ[g_\eps]} (0)\big)^2-\frac {A'x^2}2+B',
$$
for \cris{$A'$ and $B'$ positive constants of Lemma~\ref{lem:bound2psi}} \john{independent of $M$} . Then, for $M$ sufficiently large, there is a constant $D$ such that for any  $\eps \john{\in(0,1)}$ and $t>0$, we have:
\begin{equation}
\label{eq:boundphi2}
\varphi^M_\eps(t,x,v)\leq -\frac{A'}2\big(v- \cris{\JJ[g_\eps]} (t)\big)^2-\frac{A'x^2}2+ B'+ Dt.
\end{equation}
\end{lem}

\begin{proof}
The map $\phi_\eps$ introduced in the proof of Lemma~\ref{lem:bound2psi} also provides a super solution to equation~\eqref{eq:phi}, since we have:
\begin{multline*}
 \LL_\eps\phi_\eps
= D
-A'\big(v- \cris{\JJ[g_\eps]} (t)\big) \cris{\JJ[g_\eps]} (t)'
-\big(\eps a+\eps N_M'(v)+1\big)
\\
+A'x(ax-bv)
+A'(v- \cris{\JJ[g_\eps]} (t))\left( N_M(v)+x+\eps^{-1}(v- \cris{\JJ[g_\eps]} (t))\right)
\\
-A'^2x^2+\eps A'
-\eps^{-1}A'^2(v- \cris{\JJ[g_\eps]} (t))^2+A'.
\end{multline*}
Similar to the untruncated case, the contribution of the terms with $\eps^{-1}$ is nonnegative when $A'\leq1$. Moreover, for $|v|<M$ we have $N(v)=N_M(v)$, allowing to use the same inequalities as in the proof of the previous lemma. It thus only remains to show that the bound remains valid for $v>M$ or $v<-M$. This is not hard, based on the assumption that the map $N_M$ grows linearly in $v$ in that domain. However, we need to be very precise about the constant we choose. First, for $v>M$, it follows that one can find $C_{1}$, $C_{2}$ and $C_{3}$, specified later, for which:
\begin{eqnarray*}
 -A'(v- \cris{\JJ[g_\eps]} (t)) \cris{\JJ[g_\eps]} (t)' &\geq& -A'C_Iv-A'C_I^2,
\\ 
-(\eps a+\eps N'M(v)+1) &=& -\eps a-1-\eps N'(M),
\\
A'x(ax-bv) &=& A'ax^2-A'bxv \,\geq\, A'ax^2 - \frac{A'bC_1 x^2}{2}-\frac{A'b v^2}{2C_1}
\\
A'v(N_M(v)+x) & = & A'vN_M(v)+A'vx \,\geq\,A'vN_M(v) - \frac{A'C_2 x^2}{2}-\frac{A' v^2}{2C_2},
\end{eqnarray*}
and
$$
-A' \cris{\JJ[g_\eps]} (t)(N_M(v)+x)  \geq -A'C_I(N_M(v)+|x|)\geq -A'C_IN_M(v) - \frac{A'C_IC_3 x^2}{2}-\frac{A'C_I}{2C_3}.
$$
We have then that all terms in $\LL_\eps\phi_\eps$ proportional to $x$ and $x^2$ can be bounded by below by
$$
 A'(a-A')x^2 - \frac{A'bC_1 x^2}{2}- \frac{A'C_2 x^2}{2}- \frac{A'C_IC_3 x^2}{2}\geq0,
$$
by taking $C_1=\frac{2(a-A')}{3b}$, $C_2=\frac{2(a-A')}{3}$ and $C_3=\frac{2(a-A')}{3C_I}$. On the other hand, for that choice of constants, the terms proportional to $v$ are such that
\begin{multline*}
-A'C_Iv - \frac{A'bv^2}{2C_1}+A'vN_M(v) - \frac{A'v^2}{2C_2} - A'C_IN_M(v)-\eps N'(M)
\\
= -A'C_Iv - \frac{3A'b^2v^2}{a-A'}+A'vN_M(v) - \frac{3A'v^2}{a-A'} - A'C_IN_M(v)-\eps N'(M),
\end{multline*}
but term $vN_M(v)$ corresponds to a quadratic polynomial with leading term $N'(M)$. More precisely, by taking $M>2C_I$, we have that
$$
A'vN_M(v)-A'C_IN_M(v) = \frac{A'vN_M(v)}{2}+A' N_M(v)\left(\frac v2-C_I\right)\geq\frac{A'vN_M(v)}{2},
$$
and also that
$$
A'vN_M(v) =  A'vN(M)+A'vN'(M)(v-M),
$$
thus it is possible to find $M^*$ such that for all $M>M^*$:
$$
A'vN_M(v)-A'C_IN_M(v)-\eps N'(M) \geq \frac{A'vN(M)}4+\frac{A'vN'(M)(v-M)}{2}.
$$
Finally, it remains to find the sign on the following expression
\begin{multline*}
-A'C_Iv - \frac{A'bv^2}{2C_1}+A'vN_M(v) - \frac{A'v^2}{2C_2} - A'C_IN_M(v)-\eps N'(M)
\\
\geq -\frac{A'C_I^2}2-\frac{A'v^2}2 - \frac{3A'b^2v^2}{a-A'} - \frac{3A'v^2}{a-A'}+\frac{A'vN(M)}4+\frac{A'vN'(M)(v-M)}{2},
\end{multline*}
the derivative of the righthand side polynomial, when $v>M$ is such that
\begin{multline*}
 -{A'v} - \frac{6A'b^2v}{a-A'} - \frac{6A'v}{a-A'}+\frac{A'N(M)}4+\frac{A'N'(M)(2v-M)}{2} 
 \\ 
 > A'v\left(-1 - \frac{6b^2}{a-A'} - \frac{6}{a-A'}+\frac{N'(M)}{2}\right),
\end{multline*}
i.e., strictly positive as soon as
$
N'(M) >-2 - \frac{12b^2}{a-A'} - \frac{12}{a-A'}. 
$
Taking $M^*$ larger if necessary, we conclude that for all $M>M^*$:
$$
\LL_\eps\phi_\eps\geq D, \qquad \text{ for all } \john{t>0}, x\in\R\text{ and } v>M.
$$
Similar arguments hold for $v\leq -M$ and we can find $D$ large enough such that $\LL_{\eps,M}\phi_\eps(t,x,v)\geq0$. Finally, we notice that $\phi_\eps(0,x,v)=-\frac {A'}2(v- \cris{\JJ[g_\eps]} (0)^2-\frac{A'x^2}2+B\geq \varphi^M_\eps(0,x,v),$ 
so that $\phi_\eps$ is a super-solution to~\eqref{eq:phi}:
$$
 \varphi^M_\eps(t,x,v)\leq -\frac {A'}2\big(v- \cris{\JJ[g_\eps]} (t)\big)^2-\frac{A'x^2}2+B+Dt,\qquad \john{\text{ for all }t\geq 0}.
$$

\end{proof}

\section{Regularity of the truncated problem}
\label{sec:preliminary}
Equation~\eqref{eq:truncated} approximates the original nonlinear equation~\eqref{eq:evo1} by replacing the drift function $N$ by a continuously differentiable function $N_{M}$, identical to $N$ for $\vert v\vert<M$, and with linear growth at infinity. Thanks to the globally Lipschitz-continuity of this drift function, we derive uniform regularity estimates in time and space for the associated family $\{\varphi^M_\eps\}_\eps$ of Hopf-Cole transformation of the solution $f^{M}_{\eps}$ to equation~\eqref{eq:truncated}$ (i.e., \varphi^M_\eps=\eps\log f^M_\eps$). These results open the way to using the Arzel\`a-Ascoli theorem when taking the limit $\eps\to 0$.

\label{sec:regularity}

For the truncated system, the Lipschitz-continuity of the drift indeed allows us to use Bernstein's method (see e.g.~\cite{barles1991weak,barles2009concentration}) to show regularity estimates. Actually, equation~\eqref{eq:phi} has a regularizing effect, and solutions are uniformly Lipschitz continuous at any positive time, independently of the regularity of initial conditions.
 \smallskip
 
\begin{prop}
\label{prop:Repsphi}
Assume that $\eps\leq 1$ and let $E=\sqrt{B'+DT}$ for $B'$ and $D$ the two positive constants of Lemma~\ref{lem:bound2phi} and $T>0$ \john{an arbitrary time}. Noting $w_\eps$ the map $w_\eps(t,x,v):=\sqrt{2E^2-\varphi_\eps^M(t,x,v)}$, there exists a constant $\theta(M)$ independent of \john{$\eps\in (0,1)$ and $T>0$ } such that:
 $$
  |\partial_v w_\eps(t,x, v)|\leq \sqrt{\frac{\eps}{2t}}+\theta(M),\quad
  |\partial_x w_\eps(t,x, v)|\leq \sqrt{\frac{1}{2t}}+\theta(M),
 $$
 all $0<t\leq T\text{ and } (x,v)\in\R^{2}$.
\end{prop}

\begin{proof}
First, we remark that thanks to Lemma~\ref{lem:bound2phi}, $w_\eps$ is well defined, as the square root of a non-negative quantity. For an arbitrary smooth invertible map $\beta$ with non-vanishing derivative (a specific choice will be made below), define $\varphi^M_\eps(t,x,v)=\beta(w_\eps(t,x,v))$. The equation~\eqref{eq:phi} on $~\varphi^M_\eps(t,v)$ rewrites for $w_{\eps}$ as:
\begin{multline*}
\nonumber
\partial_tw_\eps
=
\frac{\left(\eps a+\eps N_M'(v)+1\right)}{\beta'(w_\eps)}
+
\big(ax-bv\big)\partial_xw_\eps
+
\left( N_M(v)+x+\eps^{-1}(v- \cris{\JJ[g_\eps]} (t))\right)\partial_vw_\eps
\\
+
\left[\beta'(w_\eps)+\eps\frac{\beta''(w_\eps)}{\beta'(w_\eps)}\right]|\partial_xw_\eps|^2+\eps\,\partial^2_{xx}w_\eps
+
\left[\eps^{-1}\beta'(w_\eps)+\frac{\beta''(w_\eps)}{\beta'(w_\eps)}\right]|\partial_vw_\eps|^2+\partial^2_{vv}w_\eps.
\end{multline*}

Define $p^v_\eps=\partial_vw_\eps$ and $p^x_\eps=\partial_xw_\eps$. To prove our result, we derive an upperbound for 
$\vert p_{\eps}\vert^{2} = \vert p^v_\eps\vert^{2} + \vert p^x_\eps\vert^{2}.$ In particular, we will control the derivative with respect to time given by
$$
\partial_t|p_\eps|^2=2 \left(p^v_\eps\, \partial_{t} p^v_\eps + p^x_\eps\, \partial_{t} p^x_\eps\right).
$$
We express each term by differentiating the above equation with respect to $v$ and $x$ respectively and rearranging the terms. 
\begin{multline*}
\partial_tp^v_\eps 
=
\frac{\eps N_M''(v)}{\beta'(w_\eps)}
-
\left(\eps a+\eps N_M'(v)+1\right)\frac{\beta''(w_\eps)}{\big|\beta'(w_\eps)\big|^2}\,p^v_\eps 
-bp^x_\eps
+
(ax-bv)\partial_xp^v_\eps
\\
+
\left( N_M'(v)+\eps^{-1}\right)p^v_\eps 
+
\left( N_M(v)+x+\eps^{-1}(v- \cris{\JJ[g_\eps]} (t))\right)\partial_vp^v_\eps 
\\
+
\Lambda(w_{\eps}) \left( p^v_\eps\, |p^x_\eps|^2 + \eps^{-1}  \big(p^v_\eps\big) ^3 \right) 
+ \Gamma(w_{\eps}) \left(\partial_{v} \vert p^x_\eps\vert^{2} +  \eps^{-1} \partial_v \vert p^v_\eps\vert^{2}\right)
+
\eps\,\partial^2_{xx}p^v_\eps 
+
\partial^2_{vv}p^v_\eps.
\end{multline*}

with
\[
\begin{cases}
\Lambda(w_{\eps})=\left[\beta''(w_\eps)+\eps\frac{\beta'''(w_\eps)}{\beta'(w_\eps)}-
\eps\left \vert\frac{\beta''(w_\eps)}{ \beta'(w_\eps)}\right \vert^2\right]
\\[10pt]
\Gamma(w_{\eps})=\left[\beta'(w_\eps)+\eps\frac{\beta''(w_\eps)}{\beta'(w_\eps)}\right]
\end{cases}
\]

Similarly, differentiating with respect to $x$, we obtain:

\begin{multline*}
\partial_tp^x_\eps
=
-
\left(\eps a+\eps N_M'(v)+1\right)\frac{\beta''(w_\eps)}{\big|\beta'(w_\eps)\big|^2}\,p^x_\eps
+ap^x_\eps
+
(ax-bv)\partial_xp^x_\eps 
\\
+
p_\eps^v
+
\left( N_M(v)+x+\eps^{-1}(v- \cris{\JJ[g_\eps]} (t))\right)\partial_vp^x_\eps 
\\
+
\Lambda(w_{\eps}) \left( \big(p^x_\eps\big)^3 + \eps^{-1}  p^x_\eps\, |p^v_\eps|^2 \right) 
+
\Gamma(w_{\eps})\left( \partial_x \vert p^x_\eps\vert^{2} + \eps^{-1} \partial_{x} \vert p^v_\eps\vert^{2}\right)
+
\eps\,\partial^2_{xx}p^x_\eps 
+
 \partial^2_{vv}p^x_\eps.
\end{multline*}
We thus obtain:
\begin{multline}
\label{eq:cacho}
\partial_t\frac{|p_\eps|^2}2
=
-
\left(\eps a+\eps N_M'(v)+1\right)\frac{\beta''(w_\eps)}{\big|\beta'(w_\eps)\big|^2}|p_\eps|^2
+
\frac 1 2 (ax-bv)\,\partial_x|p_\eps|^2
\\
+
\frac 1 2 \left( N_M(v)+x+\eps^{-1} (v- \cris{\JJ[g_\eps]} (t))\right)\partial_v|p_\eps|^2
\\
+
\left( N_M'(v)+ \eps^{-1} \right)|p^v_\eps|^2
+
 p^v_\eps\,\big(\partial^2_{vv}p^v_\eps+\eps\partial^2_{xx}p^v_\eps \big)
+
\frac{\eps N_M''(v)}{\beta'(w_\eps)}p^v_\eps\\
+
a|p^x_\eps|^2
+
 p^x_\eps\,\big( \partial^2_{vv}p^x_\eps+\eps\,\partial^2_{xx}p^x_\eps\big)
+
(1-b)\,p^v_\eps \, p^x_\eps
\\
+
\Lambda(w_\eps) \left(|p^v_\eps |^2\, |p^x_\eps|^2+ \big|p^x_\eps\big|^4 + \frac 1 \eps \left(\big|p^v_\eps\big|^4+|p^x_\eps|^2\,|p^v_\eps|^2\right)
\right)\\
+2\Gamma(w_{\eps})\left(p^v_\eps\, p^x_\eps \,\partial_xp^v_\eps +p^x_\eps\,p^x_\eps \,\partial_xp^x_\eps + \frac 1 \eps \left(p^v_\eps\, p^v_\eps \,\partial_vp^v_\eps +p^x_\eps\,p^v_\eps \,\partial_vp^x_\eps\right)
\right).
\end{multline}
By defining the operator $\Delta^\eps := \partial^2_{vv}+\eps\partial^2_{xx},$
equation~\eqref{eq:cacho} can be re-expressed as:
\begin{multline}
\label{eq:cacho2}
\partial_t\frac{|p_\eps|^2}2-\sum_{i=x,v} p^i_\eps\big(\Delta^\eps p^i_\eps\big)
=
-
\left(\eps a+\eps N_M'(v)+1\right)\frac{\beta''(w_\eps)}{\big|\beta'(w_\eps)\big|^2}|p_\eps|^2
+
(ax-bv)\partial_x\frac{|p_\eps|^2}2\\
+
\left( N_M(v)+x+\eps^{-1}(v- \cris{\JJ[g_\eps]} (t))\right)\partial_v\frac{|p_\eps|^2}2
\\
+
\left( N_M'(v)+\eps^{-1}\right)|p^v_\eps|^2
+
\frac{\eps N_M''(v)}{\beta'(w_\eps)}p^v_\eps
+ 
a|p_\eps^x|^2
+
(1-b)\,p^v_\eps\, p^x_\eps
\\
+
\Lambda(w_{\eps})\,\left(|p_\eps |^2\, |p^x_\eps|^2+\frac 1 \eps \big|p_\eps\big|^2|p_\eps^v|^2\right)
+\Gamma(w_\eps) \,\left(p^x_\eps \,\partial_x|p_\eps|^2 +\frac 1 \eps p^v_\eps \,\partial_v|p_\eps|^2\right).
\end{multline}

\noindent We now specify $\beta(w)=-w^2+2E^2$ as in the statement of the proposition. This map is smooth and smoothly invertible for $w \geq E$, and we have for such arguments:
$$
\frac{1}{|\beta'(w_\eps)|}\leq \frac1{2E},\qquad \frac{|\beta''(w_\eps)|^2}{|\beta'(w_\eps)|^2}=\frac{1}{w_\eps^2}\leq \frac{1}{E^2},\qquad\frac{|\beta''(w_\eps)|}{|\beta'(w_\eps)|^2}=\frac{1}{2w_\eps^2}\leq\frac{1}{2E^2},
$$
and $\beta'''(w_\eps)=0$. Moreover, independently of the value of $\eps$ it holds that
\[\Lambda(w_{\eps}) = -2-\frac{\eps}{w_\eps^{2}} \leq -2\]
and thus we obtain the following upper-bound for the term in $\Lambda$ in equation~\eqref{eq:cacho2}
\[
\Lambda(w_{\eps})\,\left(|p_\eps |^2\, |p^x_\eps|^2+\frac 1 \eps \big|p_\eps\big|^2|p_\eps^v|^2\right) \leq -2\left(|p_\eps^x|^4+\frac 1 \eps|p_\eps^v|^4\right).
\]

Moreover, since the truncated drift $N_M$ was constructed so that both the first and one-sided second derivatives are bounded, so for $w\geq E$ and $\eps$ sufficiently small, there exists a constant $C(M)$ independent of $\eps\leq1$ such that
\begin{multline*}
\Bigg \vert
\left(\eps a+\eps N_M'(v)+1\right)\frac{\beta''(w_\eps)}{\big|\beta'(w_\eps)\big|^2}|p_\eps|^{2}
\\
+
 \left( N_M'(v)+\eps^{-1}\right)|p_\eps^{v}|^{2}
+
\frac{\eps N_M''(v)}{\beta'(w_\eps)} p_{\eps}^{v}
+ 
a\,|p_\eps^x|^{2}
+
(1-b)p^v_\eps p^x_\eps
 \Bigg\vert
\\
\leq 
\left(
\frac{\eps a+\eps |N_M'(v)|+1}{2E^2}
+
 |N_M'(v)|
+
a+|1-b|
\right)\,| p_\eps|^2
+
\frac{\eps |N_M''(v)|}{2E}\,|p_\eps^v|
+
\eps^{-1}|p_\eps^v|^2
\\
\leq C(M) | p_{\eps}|^2+C(M)\,|p_\eps^v|+\eps^{-1}|p_\eps^v|^2.
\end{multline*}

By taking $C(M)$ larger if necessary, previous calculation leads to
\begin{multline}
\frac 1 2 \partial_t|p_\eps|^{2}-\sum_{i=x,v} p^i_\eps\big(\Delta^\eps p^i_\eps\big)
\\
\leq
 -2|p_\eps^x|\left(|p_\eps^x|^3+C(M)(1+|p_\eps^x|)\right)-2\eps^{-1}|p_\eps^v|\left(|p_\eps^v|^3+C(M)(1+|p_\eps^v|)\right)
\\
+ 
(ax-bv)\partial_x\frac{|p_\eps|^{2}}{2}
+
\left( N_M(v)+x+\eps^{-1}(v- \cris{\JJ[g_\eps]} (t))\right)\partial_v\frac{|p_\eps|^{2}}{2}
\\
+
\Gamma(w_\eps) \,\left(p^x_\eps \,\partial_x|p_\eps|^2 +\frac 1 \eps p^v_\eps \,\partial_v|p_\eps|^2\right).
\label{eq:subsol}
\end{multline}
On the other hand, for any $C$, there exists $\theta$ such that for any $x>0$, we have $(x-\theta)^{3}-2 x^{3}+C(1+x)<0$, we can find $\theta(M)$ independent of \john{$\eps\in(0,1)$ and $T>0$} such that:
\[(\vert x\vert-\theta(M))^{3}- 2 \vert x\vert^{3}+C(M)(1+\vert x\vert) \leq 0.\]
and we conclude that:
\begin{multline*}
\frac 1 2 \partial_t|p_\eps|^{2}-\sum_{i=x,v} p^i_\eps\big(\Delta^\eps p^i_\eps\big)
\leq 
-\vert p_{\eps}^x\vert \left(\vert p_{\eps}^x\vert-\theta(M)\right)^{3}
-\frac{\vert p_{\eps}^v\vert}{\eps} \left(\vert p_{\eps}^v\vert-\theta(M)\right)^{3}
\\
+
(ax-bv)\partial_x\frac{|p_\eps|^{2}}{2}
+
\left( N_M(v)+x+\eps^{-1}(v- \cris{\JJ[g_\eps]} (t))\right)\partial_v\frac{|p_\eps|^{2}}{2}
\\
+\Gamma(w_\eps) \,\left(p^x_\eps \,\partial_x|p_\eps|^2 +\frac 1 \eps p^v_\eps \,\partial_v|p_\eps|^2\right).
\end{multline*}
We can now find an upper-bound for $\vert p_{\eps}\vert$ independent of $(x,v)$ finding the positive solutions of the ordinary differential equation:
\[
z(t) = \big(z_1(t),z_2(t)\big),\quad \frac 1 2 \frac{d}{dt} |z|^{2}=-z_1 (z_1-\theta(M))^{3}-\frac{z_2}{\eps} (z_2-\theta(M))^{3}.
\]
Indeed, the map $$z(t)=\left(\sqrt{\frac{1}{2t}}+\theta(M),\sqrt{\frac{\eps}{2t}}+\theta(M)\right),$$ provides a common upper-bound with $z(0)=(+\infty,+\infty)$. Indeed, since $$(|p_\eps^x(t)|,|p_\eps^v(t)|)$$ is a sub-solution of~\eqref{eq:subsol} which is exactly solved by $z(t)$, we conclude that
$$
|p_\eps^v(t,x,v)|\leq \sqrt{\frac{\eps}{2t}}+\theta(M) ,\qquad 0<t\leq T \text{ and } (x,v)\in\R^{2}
$$
and the conclusion follows.
\end{proof}
\bigskip

\begin{lem}
\label{lem:UnifBoundphi}
 Under conditions of Proposition~\ref{prop:Repsphi}, the family $\{\varphi^{M}_\eps\}_{\john{\eps\leq 1}}$ is uniformly bounded in compact subsets of $(0,+\infty)\times\R^2$ for each fixed $M>0$.
\end{lem}
\smallskip
\begin{proof}
 We already know from Lemma~\ref{lem:bound2phi}, that $\varphi^M_\eps$ is locally upper-bounded. We show now that for all $R$ and all $0\leq t<T$ fixed, it is also lower-bounded on $[t,T]\times B_R(0)$. First, using the upper bound for $\varphi^M_\eps$ and the bound for $ \cris{\JJ[g_\eps]} $,  we have that
$$
 \varphi^M_\eps\leq -\frac {A'}2\big(v- \cris{\JJ[g_\eps]} (t)\big)^2 - \frac {A'}2x^2+ B'+ DT\leq -\frac {A'}4 (x^2+v^2)+C(T),
$$ 
for all $t\in[0,T]$, for some constant $C(T)$ independent of $R$. Therefore, for $R$ large enough there is some $\eps_0$ such that for all $\eps\leq\eps_0$ such that
  $$
\int_{B_R(0)^c}f^M_\eps= \int_{B_R(0)^c} e^{\varphi^M_\eps/\eps}\leq\int_{B_R(0)^c} e^{(-\frac{A'}4(v^2+x^2)+C(T))/\eps_0}dv<\frac12.
 $$
On the other hand, thanks to mass conservation we get that
 $$
 1=\int_{B_R(0)} e^{\varphi^M_\eps/\eps}+\int_{B_R(0)^c} e^{\varphi^M_\eps/\eps}\quad\Rightarrow\quad\int_{B_R(0)} e^{\varphi^M_\eps/\eps}\geq\frac12.
 $$
  From here, we deduce that there is some $\eps_1$ such that for all $\eps<\eps_1$
 $$
 \exists\, (x_\eps,v_\eps)\in B_R(0),\qquad \varphi^M_\eps(t,x_\eps,v_\eps)>-1,\qquad w_\eps(t,x_\eps,v_\eps)<\sqrt{2E^2+1}.
 $$
 Moreover, thanks to Proposition~\ref{prop:Repsphi} we have that $|\nabla w_\eps|$ is bounded on $[t,T]\times B_R(0)$, then we get that
 $$
 |w_\eps(t,x+h_1,v+h_2)-w_\eps(t,x,v)|\leq\Big(\sqrt{\frac 1 {2t}}+\theta(M)\Big)|h|,\quad h=(h_1,h_2),
 $$
 thus for all $(x,v)\in B_R(0)$, all $0<t<T$ and all $\eps<1$
 $$
  w_\eps < C(t,T,R):=\sqrt{2E^2+1}+2\Big(\sqrt{\frac {1} {2t}}+\theta(M)\Big)R:=C(t,T,R),
 $$
 and from the definition of $w_\eps$ we finally get that
 $$
 \varphi^M_\eps(s,x,v)>2E^2-C(t,T,R)^2>-\infty,\qquad \forall (s,x,v)\in[t,T]\times B_R(0).
 $$
\end{proof}

We finally state a result on the regularizing effect of the equation in time, which will allow taking the limit for any $M$ fixed in compact subsets of $(0,+\infty)\times\R^2$. 
\smallskip
\begin{lem}
\label{lem:Rtimephi}
For all $\eta>0$, $R>0$ and $t_0>0$, there exists a positive constant $\Theta$ such that for all $(t,s,(x,v))\in[t_0,T]\times[t_0,T]\times B_{R/2}(0)$ such that $0<t-s<\Theta$, and for all $\eps\leq1$ we have
$$
|\varphi^M_\eps(t,x,v)-\varphi^M_\eps(s,x,v)|\leq 2\eta.
$$
\end{lem}

\begin{proof}
The proof follows the same method used in~\cite[Lemma 9.1]{barles2002geometrical} (see also~\cite{barles2009concentration,benachour2009sharp} for other applications of these method). Here, we need to be careful about the constants we choose, due to the presence of a diverging term in $\eps^{-1}$. 

For any $\eta>0$, we need to find some positive constants $K$ and $C$ such that for any $(x,v)\in B_{R/2}(0)$, $s\in[t_0,T]$ and $\eps<\eps_0$, we have:
$$
\varphi^M_\eps(t,y,w)-\varphi^M_\eps(s,x,v)\leq\eta+K|x-y|^2+Ke^{\eps|v-w|^2}+C(t-s)
$$
and
$$
\varphi^M_\eps(t,y,w)-\varphi^M_\eps(s,x,v)\geq -\eta-K|x-y|^2-Ke^{\eps|v-w|^2}-C(t-s)
$$
for any $(y,w)\in B_{R/2}(0)$, $t\in[s,T]$. Both inequalities are proved in an analogous manner; we only deal with the first one. Fix some $s\in[t_0,T]$ and $(x,v)\in B_{R/2}(0)$, and define
$$
\phi_\eps(t,y,w)=\varphi^M_\eps(s,x,v)+\eta+K|x-y|^2+Ke^{\eps|v-w|^2}+C(t-s),
$$
which is well defined for $t\in[s,T)$ and $(y,w)\in B_R(0)$. Both positive constants $K,C$ will be specified later. According to Lemma~\ref{lem:UnifBoundphi}, $\varphi^M_\eps$ is locally uniformly bounded, so there is some constant $K$ large enough such that for all $\eps\leq\eps_0$,
$$
K\geq 2\|\varphi^M_\eps\|_{L^\infty([t_0,T]\times B_R(0))}.
$$
Since $e^{\eps|v-w|^2}\geq1$, the choice of $K$ allows us to write for any $(y,w)\in\partial B_{R}(0)$ that
$$
 \phi_\eps(t,y,w) \geq  \varphi^M_\eps(t,y,w)-2\|\varphi^M_\eps\|_{L^\infty([t_0,T]\times B_R(0))}+\eta+Ke^{\eps|v-w|^2}>\varphi^M_\eps(t,y,w),
$$
for all $\eta,C$. Next, we prove that taking $K$ large enough
\begin{equation}
\label{eq:auxiliarX}
\phi_\eps(s,y,w)>\varphi^M_\eps(s,y,w),\quad\text{ for all }(y,w)\in B_R(0).
\end{equation}
Indeed, if this was not the case, there would exists $\eta>0$ such that for all constants $K$ we can find $(y_{K,\eps},w_{K,\eps})\in B_R(0)$ such that
\begin{equation}
\label{eq:auxcon}
\varphi^M_\eps(s,y_{K,\eps},w_{K,\eps})-\varphi^M_\eps(s,x,v)\geq\eta+K|y_{K,\eps}-x|^2+Ke^{\eps|w_{K,\eps}-v|^2},
\end{equation}
thus
$$
1\leq 2K^{-1}\|\varphi^M_\eps\|_{L^\infty([t_0,T]\times B_R(0))},
$$
which cannot be true for $K$ large enough. Therefore for all $\eta>0$, inequality~\eqref{eq:auxiliarX} holds true, as we claimed.

Finally, notice that $\partial_t\phi_\eps=C$ and that
$$
 \partial_w\phi_\eps=2K\eps(w-v)e^{\eps(w-v)^2},\quad \partial^2_{ww}\phi_\eps = 2K\eps e^{\eps(w-v)^2}+4K^2\eps^2(w-v)^2e^{\eps(w-v)^2},
$$
and finally
$$
 \partial_y\phi_\eps = 2K(y-x),\quad \partial^2_{yy}\phi_\eps = 2K.
$$
All previous quantities are bounded by constants depending on the parameters of the system and $R$, and not on $\eps$ as soon as $\eps\leq\eps_0$. Moreover
\begin{multline*}
\partial_t\phi_\eps
-
\left(\eps a+\eps N_M'(w)+1\right)
-
(ay-bw)\partial_y\phi_\eps
\\
-
\left( N_M(w)+y+\eps^{-1}(w- \cris{\JJ[g_\eps]} (t))\right)\partial_w\phi_\eps
-
|\partial_y\phi_\eps|^2-\eps\partial^2_{yy}\phi_\eps
\\
-
\eps^{-1}|\partial_w\phi_\eps|^2-\partial^2_{ww}\phi_\eps
\geq C-C(\eps_0,C_I,R,K)
\end{multline*}
which is nonnegative if $C$ is large enough. In fact, the most delicate terms are the ones of order $\eps^{-1}$, but using the exponential part of $\phi_\eps$ we find that
$$
-\eps^{-1}(w- \cris{\JJ[g_\eps]} (t))\partial_w\phi_\eps-\eps^{-1}|\partial_w\phi_\eps|^2\sim O(1+\eps),
$$
thus depending only on $\eps_0$. Thus, $C$ can be determined depending only on the parameters of the system, $R$ and $\eps_0$. We finish by noticing that $\phi_\eps$ is a super-solution to the equation solved by $\varphi^M_\eps$ on $[s,T]\times B_R(0)$ implying that
$$
 \varphi^M_\eps(t,y,w)\leq\phi_\eps(t,y,w)=\varphi^M_\eps(s,x,v)+\eta+K|x-y|^2+Ke^{\eps|v-w|^2}+C(t-s),
$$
for all $t\in[s,T)$ and $(y,w)\in B_R(0)$. 
In a similar way, we can prove that
$$
\varphi^M_\eps(t,y,w)\geq \varphi^M_\eps(s,x,v)-\eta-K|x-y|^2-Ke^{\eps|v-w|^2}-C(t-s).
$$
To get the conclusion, take $y=x$ and $w=v$, then thanks to the previous inequalities we get that
$$
 |\varphi^M_\eps(t,x,v)- \varphi^M_\eps(s,x,v)|\leq \eta+K+C(t-s)\leq \eta+K+C\Theta,
$$
thus by taking $\Theta<(\eta+K)/C$ the result follows.
\end{proof}

\section{The large coupling limit}
\label{sec:rel}
The above results and bounds derived now allow proving the main convergence result for the truncated and non-truncated systems, and characterizing the limits of $f_{\eps}^{M}$ and $g_{\eps}$ as $\eps\to 0$ (Theorems~\ref{theo:main} and~\ref{theo:main2}). 

\subsection{Asymptotic behavior of the truncated equation}
\label{sec:asymptotic}

Using the regularity results previously demonstrated, we show the existence and characterize the limit of the families $\{\varphi^{M}_\eps\}_\eps$ and $\{f^{M}_\eps\}_\eps$ as $\eps$ is going to 0, as stated in Theorem~\ref{theo:main}. The technique again follows a somewhat classical methodology and the interested reader can find more details in~\cite{barles2009concentration} and references therein. 

\begin{proof}[Proof of Theorem~\ref{theo:main}]
We demonstrate the local convergence both in time and space, i.e. in the compact $\CC_R = [t_0,t]\times B_R(0)$ for $0<t_0<t$ and $R$ positive.

\subsection*{Step 1 (Limit)} According to Lemmas~\ref{lem:UnifBoundphi}  and~\ref{lem:Rtimephi} and Proposition~\ref{prop:Repsphi}, $\varphi^M_\eps$ are uniformly bounded and uniformly continuous in $\CC_R$. So by Arzela-Ascoli Theorem, along subsequences, $\varphi^M_\eps$ converges locally uniformly to a continuous function $\varphi^M$. 

\subsection*{Step 2 (Maximum constraint)} We now characterize this limit. First of all, we show that $\varphi^{M}$ is not lower-bounded by a strictly positive quantity. Indeed, if for some $t,x,v$ we have that $0<a\leq\varphi^{M}(t,x,v)$. By continuity of $\varphi^{M}$ there is some small $r$ such that $\varphi^{M}(t,y,w)\geq a/2$ for all $(y,w)\in B_r(x,v)$. In consequence, for $\eps$ small enough, we would have:
$$
1\geq\int_{B_r(x,v)} e^{\varphi^M_\eps(t,y,w)/\eps}dy\,dw\geq \pi r^2e^{\frac a{2\eps}}
$$
leading to a contradiction for $\eps$ small. Therefore, $\varphi^{M}(t,x,v)$ is cannot be lower-bounded by a strictly positive constant. We conclude by using again the convergence in $\CC_R$. To prove that $\max_{(x,v)\in\R^2}\varphi^{M}(t,x,v)=0$, it suffices to show that $\lim_{\eps\rightarrow0}f^M_\eps(t,x,v)\neq0$ for some $(x,v)$. From Lemma~\ref{lem:bound2phi} we have that
$$
\varphi^M_\eps(t,x,v)\leq- \frac A2\big(v- \cris{\JJ[g_\eps]} (t)\big)^2-\frac A2 x^2+ B+ DT\leq -\frac A4 (x^2+v^2)+\frac{C(T)}4,
$$
some constant $C(T)$. For $R$ large, thanks to uniform convergence, we have that
\begin{multline*}
\int_{\R^2\setminus B_R(0)}e^{\varphi^{M}(t,x,v)/\eps}\,dx\,dv= \lim_{\eps\rightarrow0}\int_{\R^2\setminus B_R(0)}f_\eps^{M}(t,x,v)\,dx\,dv
\\
\leq\lim_{\eps\rightarrow0}\int_{\R^2\setminus B_R(0)}\exp\left(-\frac{Av^2-C(T)}{4\eps}\right)dv=0.
\end{multline*}
Then, we can now focus only for $(x,v)\in B_R(0)$. In particular, from mass conservation, we have that
$$
1=\lim_{\eps\rightarrow0}\int_{B_R(0)}f^M_\eps(t,x,v)\,dx\,dv
$$
which cannot be true if $\varphi^M_\eps(t,x,v)<0$ inside the whole subset. It follows that for any $t\geq t_0$ there is some $(x,v)\in B_R(0)$ such that $\varphi^{M}(t,x,v)=0$.

\subsection*{Step 3 (Characterization of the support of the limit)} Fix some $t'\in(0,1)$ and assume that for some $(x',v')\in\R^2$ we have that $\varphi^M(t',x',v')=-a<0$. From uniform continuity in any compact $\CC_M$ containing $(t',x',v')$,  we can find a small neighborhood of $(x',v')$ such that 
$$
 \varphi^M_\eps(t,x,v)\leq-\frac a2<0,\qquad \text{ for all }t\in[t'-\delta,t'+\delta]\text{ and }(x,v)\in B_\delta(x',v'),
$$
for all $\eps\leq\eps_0$ small. Thus
$$
 \int_{B_\delta(x',v')}f^{M}(t,dx,dv) =\lim_{\eps\rightarrow0} \int_{B_\delta(x',v')}e^{\varphi^M_\eps(t,x,v)/\eps}\,dx\,dv = 0,
$$
therefore $\supp f^M(t,\cdot)\subset\{\varphi^M(t,\cdot)=0\}$ for almost every $t$.

\subsection*{step 4 (Limit equation)} \johnNew{Let us now consider $\alpha(t)$ and adherent function to the sequence $\JJ[g_\eps]$, and denote $(\eps_n)_{n\in \N}$ the associated extracted subsequence. Rewriting equation~\eqref{eq:phi} on that subsequence, we have:}
\begin{multline}
\label{eq:hct2}
\eps_n\, \partial_t\varphi^M_{\eps_n} = 
{\eps_n} \left({\eps_n} a+{\eps_n} N_M'(v)+1\right)
+
{\eps_n} (ax-bv)\partial_x\varphi^M_{\eps_n}
\\
+
\left( {\eps_n} N_M(v)+{\eps_n} x+(v- \cris{\JJ[g_{\eps_n}]} (t))\right)\partial_v\varphi^M_{\eps_n}
\\
+
{\eps_n}|\partial_x\varphi^M_{\eps_n}|^2
+
{\eps_n}^2\partial^2_{xx}\varphi^M_{\eps_n}
+
|\partial_v\varphi^M_{\eps_n}|^2
+
{\eps_n}\partial^2_{vv}\varphi^M_{\eps_n},
\end{multline}
Define $H_\eps\in C(\R^2\times\R^2)$ by
\begin{multline*}
 H_\eps(x,v,p,q) =\eps \left(\eps a+\eps N_M'(v)+1\right)
+
\eps (ax-bv)p
\\
+
\left( \eps N_M(v)+\eps x+(v- \cris{\JJ[g_\eps]} (t))\right)q
+
\eps p^2
+
q^2
\end{multline*}
which, as $\eps$ vanishes along the sequence $({\eps_n})$, converges locally uniformly towards
$$
 H(v,p) = \left( v-\alpha(t)\right)q+q^2,
$$
since $ \cris{\JJ[g_{\eps_n}]} (t)$ converges uniformly towards $\alpha(t)$ on any interval $[0,T]$ with $T>0$ \johnNew{(up to the extraction of a subsequence)}. Using that $\varphi^M_{\eps_n}$ converges also locally uniformly towards $\varphi^M$ in $(0,+\infty)\times\R^2$ is sufficient to get this conclusion. Indeed, take $(t',x',v')\in(0,+\infty)\times\R^2$ and $\phi\in C^2_b(\R_+\times\R^2)$ a regular test function, such that $(t',x',v')$ is a maximum for $\phi(t,x,v)-\varphi^M(t,x,v)$. Take some subsequences $\varphi^M_\eps$, $\phi_\eps$, $t_\eps$ and $(x_\eps,v_\eps)$ such that the same properties hold, thus at the local minima $(t_\eps,x_\eps,v_\eps)$ it holds that
\begin{multline*}
\eps\, \partial_t\varphi^M_\eps 
-
\eps (ax-bv)\partial_x\varphi^M_\eps
-
\left( \eps N_M(v)+\eps x+(v- \cris{\JJ[g_\eps]} (t))\right)\partial_v\varphi^M_\eps
\\
-
\eps|\partial_x\varphi^M_\eps|^2
-
\eps^2\partial^2_{xx}\varphi^M_\eps
-
|\partial_v\varphi^M_\eps|^2
-
\eps\partial^2_{vv}\varphi^M_\eps
\leq 
\eps \left(\eps a+\eps N_M'(v)+1\right).
\end{multline*}
Using the fact that \johnNew{$\john{\mathcal{J}[g_{\eps_n}]}(t)\rightarrow \alpha(t)$ a.e. $t\geq0$ as $n\to \infty$ (and ${\eps_n}\to0$)} and the uniform local convergence, it follows that
$$
(v'-\alpha(t'))\partial_v\phi(t',v')
-
|\partial_v\phi(t',v')|^2
\leq
0,
$$
therefore, $\varphi^M$ is a sub-solution to~\eqref{eq:limit1}. By a similar argument we can prove that $\varphi^M$ is a super-solution to the same equation and then we have shown that $\varphi^M$ is indeed a viscosity solution to~\eqref{eq:limit1}.
\end{proof}

\subsection{Relaxing the truncation assumption}\label{sec:relationship}

We now build upon these results to show that we can relax the truncation and find an analogous result for the non-truncated system, as stated in theorem~\ref{theo:main2}. So far, we know that $ \cris{\JJ[g_\eps]} (t)$ converges locally uniformly to some function $\alpha(t)$ {along subsequences}, this function will be determined later in section~\ref{sec:indentification}. Thanks to the previous section, we also know that for each $M$, the {subsequence} $\varphi_\eps^{M}$ converges locally towards a continuous function $\varphi^M$ viscosity solution to
  \begin{equation}
\nonumber 
0=(v-\alpha(t))\partial_v\varphi^M+|\partial_v\varphi^M|^2,\quad \max_{(x,v)\in\R^2}\varphi^M(t,x,v)=0,\quad \text{ for all }t\in(0,T).
\end{equation}
Now, take two truncation parameters $M, K$. Summarizing our results, we can define $f^M$ and $f^K$ as the weak limits related to $\varphi^M$ and $\varphi^K$ respectively. \

\begin{lem}
\label{lem:Reps}
There exists some positive constant $M$ large enough, such that for any regular test function $\phi\in C_c^\infty(\R^2)$:
$$
 \Big|\int_{\R^2} f^{M}_\eps(t,x,v)\phi(x,v)-\int_{\R^{2}} f^{K}_\eps(t,x,v)\phi(x,v) \Big|\rightarrow 0,
  $$
for all $K>M$ when $\eps\rightarrow0$. \john{Moreover, we have
\[ \lim_{\eps\to 0} \Big|\JJ[f^{M}_\eps(t)]-\JJ[f^{K}_\eps(t)] \Big| = 0.\]}
\end{lem}

\begin{proof}
From lemma~\ref{lem:bound2phi} we know that for any $M>0$ there exists some positive constant $C_M(T)$ such that
$$
\varphi^M_\eps(t,x,v)\leq -\frac A4\big(v^2+x^2\big)+C_M(T).
$$
For $M>M^*$ the constant $C_M(T)$ can be chosen as depending only on $M^*$ and not on $M$, thus we find that uniformly on $M>M^*$ previous inequality holds true. In detail, for any $K,M>M^*$, these remark allow us to write
$$
 \max\left\{\varphi^M_\eps(t,x,v),\varphi^K_\eps(t,x,v)\right\}\leq -\frac A4\big(v^2+x^2\big)+C_{M^*}(T),
$$
implying that there is some $R>M^*$ such that
\begin{equation}
\label{eq:cool}
 \max\left\{\varphi^M_\eps(t,x,v),\varphi^K_\eps(t,x,v)\right\}\leq -\alpha(x^2+v^2),\quad \forall\,(x,v)\in\R^2\setminus B_R(0),
\end{equation}
for some $\alpha$ small not depending on $\eps$. The rest of the proof is based on the fact that for any $K>M=R+1$, the equations solved by $f_\eps^M$ and $f_\eps^K$ are the same inside the ball $B_M(0)$. Indeed, for any $(x,v)\in B_M(0)$, we have that $h_\eps^{M,K}:=f_\eps^M-f_\eps^K$ solves
\begin{multline*}
 \partial_th_\eps^{M,K}=\big(a+N'(v)+\eps^{-1}\big)h_\eps^{M,K}+(ax-bv)\partial_xh_\eps^{M,K}
 \\
 +\big(N(v)+x+\eps^{-1}(v- \cris{\JJ[g_\eps]} (t))\big)\partial_vh_\eps^{M,K}
 +
  \eps\partial^2_{xx} h_\eps^{M,K}
 +\partial^2_{vv} h_\eps^{M,K},
\end{multline*}
then, thanks to the Kato inequality
\begin{multline*}
 \frac{d}{dt}\int_{\R^2} |h_\eps^{M,K}|\phi\leq\int_{\R^2}\big(a+N'(v)+\eps^{-1}\big)|h_\eps^{M,K}|\phi-\int_{\R^2}\partial_x\left((ax-bv)\phi\right)|h_\eps^{M,K}|
 \\
 -\int_{\R^2}\partial_v\left(\big(N(v)+x+\eps^{-1}(v- \cris{\JJ[g_\eps]} (t))\big)\phi\right)|h_\eps^{M,K}|
 \\
 +
  \eps\int_{\R^2}\phi\,\partial^2_{xx} |h_\eps^{M,K}| +\int_{\R^2}\phi\,\partial^2_{vv} |h_\eps^{M,K}|.
\end{multline*}
If we take a function $\phi$ equals to $1$ inside $B_{M-1}(0)$ and 0 outside $B_{M}(0)$ then we get
\begin{multline*}
 \frac{d}{dt}\int_{\R^2} |h_\eps^{M,K}|\phi\leq -\int_{\R^2}(ax-bv)\partial_x\phi|h_\eps^{M,K}|
 \\
 -\int_{\R^2}\big(N(v)+x+\eps^{-1}(v- \cris{\JJ[g_\eps]} (t))\big)\partial_v\phi |h_\eps^{M,K}|
 \\
 +
  \eps\int_{\R^2}\partial^2_{xx}\phi\, |h_\eps^{M,K}| +\int_{\R^2}\partial^2_{vv}\phi\, |h_\eps^{M,K}|
  \leq \eps^{-1}C\int_{\R^2\setminus B_{R}(0)} |h_\eps^{M,K}|,
\end{multline*}
for some constant $C=C(R,\phi)$. Finally, thanks to~\eqref{eq:cool} we have that
$$
 \int_{\R^2\setminus B_{R}(0)} f_\eps^M \leq \int_{\R^2\setminus B_{R}(0)}e^{-\alpha(x^2+v^2)/\eps} =2\pi \int_{R}^\infty e^{-\alpha\rho^2/\eps}\rho d\rho = \frac{\pi\eps}{\alpha}e^{-\alpha R^2/\eps},
$$ 
thus
$$
 \frac{d}{dt}\int_{\R^2} |h_\eps^{M,K}|\phi\leq \eps^{-1}C(R,\phi)\int_{\R^2\setminus B_{R}(0)} |h_\eps^{M,K}|\leq \frac{2\pi C}{\alpha}e^{-\alpha R^2/\eps},
$$
integrating and using that $h^{K,M}_\eps(0,x,v)\equiv0$ we conclude that $\int_{\R^2} |h_\eps^{M,K}|\phi\rightarrow0$ as $\eps\rightarrow0$. In particular, we get that
\begin{equation}
\label{eq:cool2}
 \int_{B_R(0)} |h_\eps^{M,K}|=\int_{B_R(0)} |f_\eps^{M}-f_\eps^{K}|\xrightarrow{\eps\rightarrow0}0.
 \end{equation}
To complete the proof, we consider an arbitrary test function $\hat\phi$, and notice that
$$
\Big|\int_{\R^2} f^{M}_\eps\hat\phi-\int_{\R^2} f^{K}_\eps\hat\phi \Big| \leq \|\hat\phi\|_\infty \int_{B_R(0)^c}(f^{M}_\eps+f^{K}_\eps) +\|\hat\phi\|_\infty \int_{B_R(0)}|h^{M,K}|,
$$
equipped with~\eqref{eq:cool} and~\eqref{eq:cool2}, the conclusion follows.
\john{Consider now $\phi=v$. The same inequality shows that:
$$
\Big|\int_{\R^2} v \, f^{M}_\eps-\int_{\R^2} v\,f^{K}_\eps \Big| \leq  \int_{B_R(0)^c} 2\,|v| \; e^{-\frac{\alpha}{\eps} (x^2+v^2)} +R \int_{B_R(0)}|h^{M,K}|,
$$
again converging to $0$ as $\eps\to 0.$
}

\end{proof}

By a very similar argument, we can now prove the following lemma: 
 \begin{prop}
 \label{prop:Reps}
There exists some $M$ large enough, such that for any regular test function $\phi\in C_c^\infty(\R^2)$ it holds that
 $$
 \Big|\int_{\R^2} f^{M}_\eps(t,x,v)\phi(x,v)-\int_{\R^2} g_\eps(t,x,v)\phi(x,v)\Big|\xrightarrow{\eps\rightarrow0}0.
  $$
  and 
  \[\Big| \JJ[f^M_\eps] - \JJ[g_\eps]\Big|\xrightarrow{\eps\rightarrow0}0.\]
\end{prop}

\johnNew{Lemma~\ref{lem:Reps} ensures that, for $M,K$ sufficiently large, the limits of $f^M$ and $f^K$ are identical. Therefore, in the weak (measure) sense, the sequence $g_\eps$ is converging towards $f^{M}$, therefore (by approximating $v$ with compactly supported functions) it holds that
$$
  \cris{\JJ[g_\eps]} (t)=\int_{\R^2} vg_\eps(t,x,v)\xrightarrow{\eps\rightarrow0} \int_{\R^2} vf^M(t,dx,dv),\quad \text{  a.e. in time,}
$$
moreover since $ \cris{\JJ[g_\eps]} $ is converging to $\alpha$ a continuous function, we have characterized the limit of $g$ in terms of the limit function $\varphi^M$.}

\section{Characterization of the limit, clamping and periodic solutions}\label{sec:indentification}
Altogether, the results of the previous sections show that the solutions to the mean-field FitzHugh-Nagumo system behave heuristically, in the large coupling limit, as $e^{-(v-\alpha(t))^{2}/2\eps}$ (in a logarithmic sense), where $\alpha$ is an adherent point of $\JJ[g_\eps]$, and therefore concentrate around singular Dirac measure on the voltage variable $v$, centered at a time-varying point $\alpha(t)=\int_{{\mathbb{R}^{2}}} v g(t,x,v)dxdv$. To complete the identification of the limit, we now characterize in this section the map $\alpha(t)$. We show in section~\ref{sec:derivation} that this map is the solution of a nonlinear ordinary differential equation in two dimensions, identical to the initial single-neuron model, and characterize the possible limiting behaviors.

We then turn in section~\ref{sec:numerics} to the numerical investigation of the system for large but finite $\eps$, illustrating the validity of the approach away from bifurcations, but also singular behaviors associated to the slow-fast nature of the perturbation in the vicinity of bifurcations. 

\subsection{Identification of a limit}\label{sec:derivation}
In this section, we propose one non-trivial limiting distribution of the process as $\eps\to 0$. We recall that the Cole-Hopf transform $\psi_{\eps}$ of the solution to the FitzHugh-Nagumo mean field equation $g_{\eps}$, given by $\psi_{\eps}=\eps\log(g_{\eps})$, satisfies the equation:
\begin{multline}\label{eq:PsiRecall}
\partial_t\psi_\eps
=
\left(\eps a+\eps N'(v)+1\right)
+
(ax-bv)\partial_x\psi_\eps
\\
+
\left( N(v)+x+\eps^{-1}(v- \cris{\JJ[g_\eps]} (t))\right)\partial_v\psi_\eps
\\
+
|\partial_x\psi_\eps|^2+\eps\partial^2_{xx}\psi_\eps
+
\eps^{-1}|\partial_v\psi_\eps|^2+\partial^2_{vv}\psi_\eps.
\end{multline}
This equation involves terms of distinct order as $\eps\to 0$: diverging terms, terms of order 1 and vanishing terms, imposing, using regular perturbations theory, the following conditions on non-vanishing terms:
\begin{equation}\label{eq:perturbations}
\begin{cases}
0=(v- \cris{\JJ[g_\eps]} (t))\partial_v\psi_\eps + |\partial_v\psi_\eps|^2\\
\partial_t\psi_\eps = 1 + (ax-bv)\partial_x\psi_\eps + ( N(v)+x)\partial_v\psi_\eps + |\partial_x\psi_\eps|^2+\partial^2_{vv}\psi_\eps. 
\end{cases}
\end{equation}
We study each equation separately, the first one characterizing the dependence of $\psi_{\eps}$ on the voltage variable and studied in section~\ref{sec:solV} and~\ref{sec:char}, the second one allowing, based on these results, to propose a full solution (characterizing the dependence in $x$ also) in the limit $\eps\to 0$ as we show in section~\ref{sec:solX}. 

\subsubsection{Dependence in the voltage variable}\label{sec:solV}
The first equality in equations \eqref{eq:perturbations} provides conditions on the dependence of $\psi_{\eps}$ in $v$ at leading order (terms of order $\mathcal{O}(\eps^{-1})$ in eq.~\eqref{eq:PsiRecall}), since it only depends on derivatives with respect to that variable. We shall denote $\psi_\eps^0$ the leading order terms of the system. Maps $\psi_{\eps}^0$ independent of $v$ are solutions of the problem, in particular $\psi_{\eps}=c\leq 0$ satisfy the equation and the negativity constraint. 
Non-constant solutions for which the differential does not vanish (except at isolated points) satisfy the equation:
\[\partial_{v}\psi_{\eps}^0=-(v-\john{\JJ[g_\eps](t)}),\]
imposing the fact that:
\begin{equation}\label{eq:psi_global}
\psi_{\eps}^0(t,x,v)=-\frac{(v-\john{\JJ[g_\eps](t))}^{2}}{2} + \phi_{\eps}(x,t)
\end{equation}
for some function $\phi_{\eps}$ to be determined. 

We now discuss the set of possible solutions of the equation. First, initial conditions satisfying the assumptions of lemma~\ref{lem:bound2psi}, we necessarily have a quadratic upperbound on $\psi_{\eps}$, a condition that cannot be satisfied for all $v$ by constant solutions. Therefore, possible solutions of the equation may either have a quadratic dependence in $v$ given by equation~\eqref{eq:psi_global}, or solutions of the type:
\[\bar{\psi}_{\eps,\gamma,\delta}^0(t,x,v)=\begin{cases}
-\frac{(v-\john{\JJ[g_\eps](t)})^{2}}{2} + \frac{(\delta-\john{\JJ[g_\eps](t)})^{2}}{2}+ \bar{\phi}_{\eps,\delta,1}(x,t) & v<\delta\\
0 & v\in [\delta,\gamma] \\
-\frac{(v-\john{\JJ[g_\eps](t)})^{2}}{2} + \frac{(\gamma-\john{\JJ[g_\eps](t)})^{2}}{2} + \bar{\phi}_{\eps,\gamma,2}(x,t) & v<\gamma.\\
\end{cases}\]

The fact that $\psi_{\eps}\leq 0$ imposes also a negativity constraint on $\phi_{\eps}$. Indeed, evaluating the negativity condition along the parametric curve $v=\john{\alpha}(t)$ for all $t$, we find that necessarily,
\[\phi_{\eps}(x,t)\leq 0.\]
For the quadratic solution~\eqref{eq:psi_global}, the only zeros of that map occur when $v=\john{\alpha}(t)$, and the solutions concentrate, in the voltage variable, as a Dirac mass at $\john{\alpha}(t)$, which we will characterize in section~\ref{sec:char}.

\subsubsection{Concentration on the deterministic FitzHugh-Nagumo equation}\label{sec:char}

The above derivations show that in the limit $\eps\to 0$, the distribution of the solutions to the mean-field FitzHugh-Nagumo equations have a voltage that may concentrate around a Dirac measure at a time-varying position $\john{\alpha}(t)$. We now derive the equation on $\john{\alpha}(t)$. Actually, we show that, given an initial condition $f(0,x,v)$ with sufficient integrability property to the mean-field FitzHugh-Nagumo equation, the map $t\mapsto \john{\alpha}(t)$ is given \cris{by} the unique solution to the ordinary (single-neuron) FitzHugh-Nagumo equation:
\begin{equation}\label{eq:LimitFhN}
\begin{cases}
\frac{d\john{\alpha}}{dt} & = -N(\john{\alpha}) -\john{\beta}\\
\frac{d\john{\beta}}{dt} & = -a \john{\beta} + b \john{\alpha}\\
\end{cases}
\end{equation}
with initial conditions $[\alpha(0),\beta(0)]=\int_{{\mathbb{R}^{2}}} [v,x] f(0,x,v)dxdv$.

To show this property, we write down the equation satisfied by $\alpha(t)$ in the limit $\eps\to 0$, starting from its original definition
\[\alpha(t) = \lim_{\eps\to 0} \int_{{\mathbb{R}^{2}}} v g_{\eps}(t,x,v)dxdv,\] 
and using the evolution equation~\eqref{eq:evo1}. Formally using the concentration of the solutions around the Dirac mass at $\alpha(t)$ in $v$, we obtain
\begin{align*}
\frac{d\alpha}{dt} & = \lim_{\eps\to 0}\Bigg[\int_{\mathbb{R}^{2}} v\partial_{v}\Big[(N(v)+x+\varepsilon^{-1}(v-\alpha))g_{\eps}(t,x,v)-\partial_{v}g_{\eps}(t,x,v)\Big] \\
& \qquad - \int_{\mathbb{R}^{2}} v\partial_{x}\Big[(ax-bv)g_{\eps}(t,x,v)+\eps \partial_{x}g_{\eps}(t,x,v)\Big]\Bigg]\\
&=\lim_{\eps\to 0}\Big[\int_{\mathbb{R}^{2}} (-N(v)-x+\varepsilon^{-1}(\alpha-v))g_{\eps}(t,x,v)-\partial_{v}g_{\eps}(t,x,v)\Big] \\
&=-N(\alpha) -\langle x\rangle
\end{align*}
using integration by parts, sufficient decay of $g_{\eps}(t,x,v)$ and the shrinkage of the support of $g_{\eps}(t,x,v)$ at $\alpha(t)$, and denoting $\langle x\rangle$ the average value of $x$ at time $t$, 
\[\langle x\rangle(t)=\int_{\R^{2}}x \;g_{\eps}(t,x,v).\] 
This quantity satisfies the equations:
\begin{align*}
\frac{d\langle x\rangle}{dt} & = \lim_{\eps\to 0}\Bigg[\int_{\mathbb{R}^{2}} x\partial_{v}\Big[(N(v)+x+\varepsilon^{-1}(v-\alpha))g_{\eps}(t,x,v)
\\
&\qquad-\partial_{v}g_{\eps}(t,x,v)\Big] - \int_{\mathbb{R}^{2}} x\partial_{x}\Big[(ax-bv)g_{\eps}(t,x,v)\Big]\Bigg]\\
&=\lim_{\eps\to 0}\int_{\mathbb{R}^{2}} (ax-bv)g_{\eps}(t,x,v) \\
&=a \langle x\rangle -b \alpha,
\end{align*}
concluding formally the argument that $\alpha(t)$ satisfies the FitzHugh-Nagumo equation, coupled to the average value of $x$:
\[
\begin{cases}
\frac{d\alpha}{dt} & = -N(\alpha) -\beta\\
\frac{d\beta}{dt} & = -a \beta  +b \alpha.\\
\end{cases}
\]

In these formal derivations, we have assumed that for regular polynomial maps $F$, the quantity $\int F(v)g_{\eps}(t,x,v)$ converges towards $F(\alpha(t))$ when $\eps\to 0$, and, more boldly, that the interaction term (with a coefficient $\eps^{-1}$) converged towards $0$. 

We provide here a non-rigorous justification of these limits, which relies on assumptions we highlight below. If the solution concentrate as predicted by the quadratic solution derived at leading order, then we have \[g_{\eps}(t,x,v)=\exp(\john{-(v-\JJ[g_\eps](t))^2/\eps} + \phi_{\eps}(t,x)/\eps + \xi_\eps(t,x,v) )\] 
where the first two terms characterize the leading order behavior of the system as $\eps \to 0$ introduced in the previous section, and $\xi_\eps$ the rest:
\[\xi_{\eps}(t,x,v)=\psi_\eps(t,x,v)+\frac{(v-\JJ[g_\eps](t))^2}{2} - \phi_{\eps}(t,x),\]
which we expect to be negligible compared to the first order terms as $\eps\to 0$, and assume that these are of order $\eps$ based on the perturbation equations derived above. We denote $\xi_\eps = \eps \zeta_\eps$, with $\zeta_\eps$ of order $1$, and drop the index $\eps$ for simplicity in the ansatz below. 
Because of the regularity of $g_\eps$, we assume in the sequel that $\zeta$ is continuously differentiable in $v$. Based on these assumptions, we obtain, because of the normalization property, and using $g_\eps=e^{\psi_\eps/\eps}$ and the change of variables $u=(v-\alpha)/\sqrt{\eps}$,
\begin{align*}
1=\int_{\R^{2}} g_{\eps}(t,x,v)\,dxdv = \sqrt{\eps} \int_{\R^{2}} e^{\zeta(t,x,\alpha(t)+\sqrt{\eps}u)}e^{-u^{2}/2 + \phi_{\eps}(t,x)/\eps} \,dxdv \end{align*}
so that for all $t\geq 0$, 
\[\lim_{\eps \to 0} \sqrt{\eps} \int_{\R^{2}} e^{ \zeta(t,x,\alpha(t))+ \frac 1 \eps \phi_{\eps}(t,x)} dx=1 \]

Let now $F$ be a polynomial of order $k$ less than $3$ (in the above derivations, only the cubic map $N$ and linear functions were involved, yet the properties below are valid for a larger class of functions). We have:
\begin{multline*}
\int_{\R^{2}} F(v) g_\eps(t,x,v)dxdv 
\\
=\sqrt{\eps} \int_{\R^{2}} F(\alpha(t)+\sqrt{\eps}u) e^{ \zeta(t,x,\alpha(t)+\sqrt{\eps}u)}e^{-u^{2}/2 + \phi_{\eps}(t,x)/\eps} \,dxdv
\\=F(\alpha(t))\sqrt{\eps} \int_{\R^{2}} e^{ \zeta(t,x,\alpha(t))+\frac 1 \eps \phi_{\eps}(t,x)} dx + O(\sqrt{\eps})\to F(\alpha(t)).
\end{multline*}
This remark justifies the convergence of $\int_{\R^{2}} [N(v),v] g_{\eps}(t,x,v)$ towards $[N(\alpha),\alpha]$. The convergence of $\eps^{-1}\int (v-\alpha) f$ towards 0 remains to be justified, and has to be handled with care because of the presence of the factor $\eps^{-1}$. Actually, the convergence to $0$ of the interaction term stems not from the fact that $f$ converges towards a Dirac measure at $\alpha(t)$, but rather because the symmetry of $\psi$ around $\alpha(t)$. Indeed, we have:
\[\frac 1 \eps \int_{\R^{2}} (v-\alpha) e^{\frac{-(v-\alpha)^{2}}{\eps}}e^{\frac 1 \eps \phi_{\eps}(t,x) + \zeta(t,x,v)} dxdv = \int_{\R^{2}} u e^{-u^{2}}e^{\frac 1 \eps \phi_{\eps}(t,x) + \zeta(t,x,\alpha(t)+\sqrt{\eps}u)} dxdu,\]
and we thus conclude that, as $\eps\to 0$:
\begin{multline*}
\frac 1 \eps \int_{\R^{2}} (v-\alpha) e^{\frac{-(v-\alpha)^{2}}{\eps}}e^{\phi_{\eps}(t,x) +\zeta(t,x,v)} dxdv 
\\
\simeq \int_\R e^{\frac 1 \eps \phi_\eps(t,x)+\zeta(t,x,\alpha(t)} \int_{\R} u e^{-u^{2} + \sqrt{\eps} \partial_v \zeta(t,x,\alpha(t))}\,du\,dx.
\end{multline*}
While the first term is of order $\eps^{-1/2}$, the second term is always equal to $0$ because $\psi_\eps(t,v)$ is a symmetric distribution around $\alpha(t)$, and thus the above integral should be indeed vanishing.

\subsubsection{Dependence in the adaptation variable}\label{sec:solX}
To characterize $\phi_{\eps}$, it is natural to inject the expression \eqref{eq:psi_global} into the characteristic equation of $\psi_{\eps}$ or simply in the second equality of equation~\eqref{eq:perturbations}. However, this leads to the equation:
\[-(v-\john{\JJ[g_\eps]}(t))\frac{d\john{\JJ[g_\eps]}(t)}{dt} + \partial_{t}\phi_{\eps} = -(N(v)+x)(v-\john{\JJ[g_\eps]}(t))+ (ax-bv)\partial_{x}\phi_{\eps} + \vert \partial_{x}\phi_{\eps}\vert^{2},\]
which cannot be solved for a $\phi_{\eps}$ independent of $v$ (a particular difficulty arises with the term depending on the product of $x$ and $v$). However, since we are considering solutions that are probability distributions, we only consider maps $\phi_{\eps}$ for which the equation is valid on the support of $g_{\eps}$, which, as we indicated above, is restricted in the limit $\eps\to 0$ to a single, time-varying point $v=\alpha(t)$. In the limit $\eps\to 0$, we can thus consider solutions with this value of $v$ to determine $\phi_{0}$ the limit of $\phi_{\eps}$ at $\eps=0$, which significantly reduces the equation to:
\[\partial_{t}\phi_{0} = (ax-b\alpha)\partial_{x}\phi_{0} + \vert \partial_{x}\phi_{0}\vert^{2}.\]
To find a particular solution to this equation consistent with our quadratic upper-bound derived in lemma~\ref{lem:bound2phi}, we start looking for possible solution of the form 
\[\phi_{0}(x,t)=\frac A 2 \; (x-\beta(t))^{2}+B\] 
for some constants $(A,B)\in \R^{2}$ and $\beta:\R^{+}\mapsto \R$ a differentiable function. A necessary condition for that function to be a solution of this equation is thus:
\[-A(x-\beta(t))\frac{d\beta}{dt} = A(ax-b\alpha)(x-\beta(t)) + A^{2}(x-\beta(t))^{2}\]
and therefore:
\[b\alpha + A\beta(t) -\frac{d\beta}{dt} = (a+A)x.\]
This equation completely characterizes the parameters of the quadratic ansatz. Indeed, because the lefthand side is independent of $x$, this requires to cancel the dependence in that variable in the lefthand side of the equation, thus implying $A=-a$. Moreover, this further implies that the map $\alpha(t)$ satisfies the differential equation depending upon $I(t)$:
\[ \frac{d\beta}{dt} = b\alpha(t) -a \beta(t),\]
providing an expression of $\beta$ only depending on $\alpha$ and the parameters of the system. 

\bigskip

We thus conclude that a solution to the equation is given by the formula:
\[\psi=-\frac 1 2 (v-\alpha(t))^{2} -\frac a 2 (x-\beta(t))^{2} + B.\]
The constraint of non-positivity of $\psi$ and the fact that it reaches $0$ imply that necessarily $B=0$, thus yielding the following form for $\psi_{\eps}$ in the limit $\eps\to 0$:
\begin{equation}\label{eq:psifinal}
\cris{\psi(t,x,v)}=-\frac 1 2 (v-\alpha(t))^{2} -\frac a 2 (x-\beta(t))^{2}.
\end{equation}

\begin{rem}
We note that the upper-bound derived in lemma~\ref{lem:bound2phi} is valid here, and provides a justification for the assumption that the constant $A$ in that lemma is such that $A\leq \min\{1,a\}$. 
\end{rem}

 By showing that the voltage variable concentrates at all times towards the a Dirac measure at a time-varying point $\alpha(t)$ solution of the FitzHugh-Nagumo equation with initial conditions equal to the average initial voltage and recovery variable, we conclude that:
\begin{enumerate}
\item The limiting behavior of the system only depends on the first moment of the initial condition, i.e. the dynamics of system~\eqref{eq:evo1} with initial conditions having the same two first moments collapse on the same solution when $\eps\to 0$;
\item Despite the presence of noise (with a fixed, non-vanishing diffusion coefficient), the limiting voltage converges towards a deterministic function of time: strong connectivity completely cancels the effects of noise.
\item The mean-field equation may have multiple stationary solutions as well as periodic solutions. 
\end{enumerate}

\subsection{Numerical simulation of the convergence}
We confirm here numerically the convergence of the distribution of the solution of the network equation towards the predicted solution. \john{To this purpose and because of the simple and explicit form of the limit obtained, we do not need to use sophisticated methods for integrating McKean-Vlasov equations, as those developed in the literature of numerical kinetic theory (see e.g.~\cite{filbet} and references therein).}   To this end, we simulated extensively the network associated with the PDE analyzed in the paper:
\begin{equation}\label{eq:networksims}
\begin{cases}
\displaystyle{dv^{i}_{t}=\left(-N(v^{i}_{t})-x^{i}_{t}+\frac 1 {\eps\,n}\sum_{j=1}^{n} (v^{j}_{t}-v^{i}_{t})\right)\,dt + \sqrt{2}\;dW^{i}_{t}}\\
dx^{i}_{t}=\left(-a\,x^{i}_{t}+b\,v^{i}_{t}\right)\,dt+ \sqrt{2\eps}\;dB^{i}_{t}
\end{cases}
\end{equation}
where $(W^{i}_{t},B^{i}_{t},i=1\cdots n)$ are independent standard Brownian motions. The theory predicts that we have propagation of chaos (see~\cite{mischler2016kinetic}) and that the distribution of the voltage and adaptation variables $(v^{i}_{t},x^{i}_{t})$ of a given neuron $i$ behaves, at leading exponential order, as independent Gaussian variables with density (making explicit the leading-order normalization term):
\[\frac{1}{\sqrt{2\pi \eps}} \exp\left(-\frac{(v-\alpha(t))^{2}}{2\eps}\right) \sqrt{\frac{a}{2\pi \eps}}\exp\left(- \frac{a\,(x-\beta(t))^{2}}{2\eps}\right).\] 
\begin{figure}[h]
\centerline{\includegraphics[width=.95\textwidth]{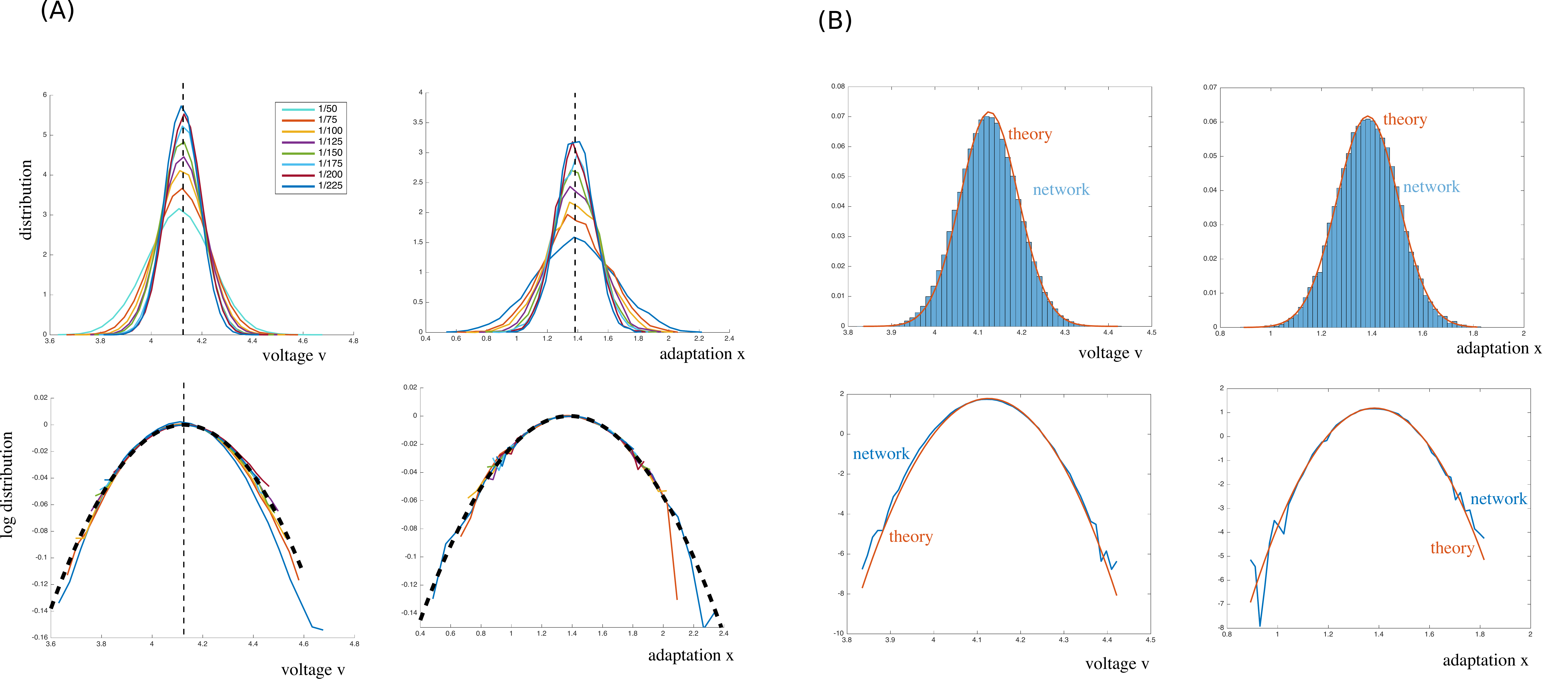}}
\caption{Concentration of measure: Simulations of the network equations with $n=5\,000$ neurons,  $N(v)=v(v-\lambda)(v-1)$ with $\lambda=4$, $a=0.3$ and $b=0.1$. (A, top) as $\eps$ decreases, the distribution of voltage (left) and adaptation variable (right) converge towards the solution of the deterministic FhN equation (dashed line). (A, bottom) An excellent match with the predicted Gaussian profile (black dashed line) is observed. Solid lines depict $\eps\log(\hat{\mu}_{v}) +\frac 1 2 \log(2\pi\eps)$ (left) or $\eps\log(\hat{\mu}_{x}) +\frac 1 2 \log(2\pi\eps/a)$ for $\hat{\mu}_{v}$ (resp., $\hat{\mu}_{x}$) the empirical distribution of $(v^{i},i\in \{1,\cdots,n\})$ (resp., $(x^{i},i\in \{1,\cdots,n\})$). (B) highlights the empirical distribution of voltage (left) and adaptation (right) for $\eps^{-1}=225$ and its good match with the predicted leading order behavior (red), in linear (top) or logarithmic (bottom) scale. }
\label{fig:Distributions}
\end{figure}

To test the accuracy of this result for finite networks and non-zero $\eps$, we computed numerically using the Euler-Maruyama scheme the solution of equation~\eqref{eq:networksims} with $n=5\,000$ neurons and for various values of $\eps$ (very similar results are already valid at smaller values of $n$). Figure~\ref{fig:Distributions} shows, for a given parameter set, the concentration of the distribution around $\alpha(t)$ (dashed line) as $\eps$ decreases (panels (A)), both for the voltage and for the adaptation variable. Moreover, a clear Gaussian profile emerges, with the parameters of the theoretical distribution, as visible in the bottom row of panel (A), where the logarithm of the distribution in $v$ or $x$ are plotted for various values of $\eps^{-1}$ shows a clear collapse on the predicted profiles $-(v-\alpha(t))^{2}/2\eps$ and $-a(x-\beta(t))^{2}/2\eps$ (dashed black curves). Panels (B) also highlight the particularly good match of the network simulation with the solution for a fixed value of $\eps^{-1}=225$.

\subsection{Some properties of the solution}
We now recall some well-know results on the solutions to the equation that govern the time-dependent points at which the solutions concentrate. The classical FitzHugh-Nagumo equation corresponds to the case where $N(x)=x(x-1)(x-\lambda)+I_{0}$ where $I_{0}$ is an input current. In that case, we have the following well-known characterization, summarized in the following:
\begin{lemma}\label{thm:FhN}
The limit equation displays either:
\begin{itemize}
\item Two stable and one unstable stationary distributions when $\Delta>0$
\item A single stationary distribution centered, in the voltage variable, at a value $v^{*}$ when $\Delta \leq 0$, which is:
\begin{itemize}
\item stable when $T<0$ and globally attractive,
\item unstable when $T>0$, in which case a globally attractive periodic solution exists;
\end{itemize}
\end{itemize}
where
\[\Delta = -27 I_{0}^{2} + 18 (1+\lambda)(\lambda+\frac {b}{a}) I_{0} - 4 (\lambda+\frac {b}{a})^{3}-4( 1+\lambda)^{3}I_{0}+( 1+
\lambda)^{2} (\lambda+\frac b a)^{2}
\]
and
\[T=-3 (v^{*})^{2}+2(1+\lambda)v^{*}-\lambda+a.\]
A saddle-node bifurcation arises at $\Delta=0$, and a Hopf bifurcation occurs at $T=0$. 
\end{lemma}

The above result stems from the classical analysis of the FitzHugh-Nagumo system. The originality of this result is that despite the fact that the PDE system has a diffusive term, the mean-field equation may display multiple stationary solutions or even no stable stationary solution and a periodic solution. This surprising phenomenon is essentially due to the nonlinear nature of the McKean-Vlasov system, and joins other similar observations in related systems~\cite{scheutzow1985some,scheutzow1988periods,scheutzow1993stabilization,scheutzow1988stationary,scheutzow1985noise,scheutzow1986periodic,tugaut2013self,tugaut2014self,herrmann2012self,tugaut2013convergence}. 

The above result is only asymptotic, and it remains an open problem to show that similar multi-stable or periodic solutions exist when $\eps$ is finite. In the case where there exists multiple solutions, the fact that these are well separated ensures that, for large but finite $\eps$, multiple stationary solutions may exist. Proving these results may require the use of spectral theory (multi-stability) or the existence of invariant hyperbolic manifolds in the flavor of~\cite{giacomin2012transitions}. 

\section{Numerical simulations}
\label{sec:numerics}

The results proved in this paper are valid for the mean field equation (limit $n\to \infty$) in the strong coupling limit ($\eps \to 0$). They may thus be of little relevance to describe the dynamics of the finite neural network~\eqref{eq:network} with finite coupling. Moreover, we argued that the solutions of equation~\eqref{eq:network} shall remain close to those with a small diffusion on the adaptation variable. In this section, we display numerical simulations of the network equation~\eqref{eq:network} and investigate whether the limit equation faithfully represents the dynamics of finite networks with noise, both in multi-stable and oscillatory regimes, and explore irregular dynamics occurring at the transition between those regimes. 

\begin{figure}[!h]
\centerline{\includegraphics[width=.95\textwidth]{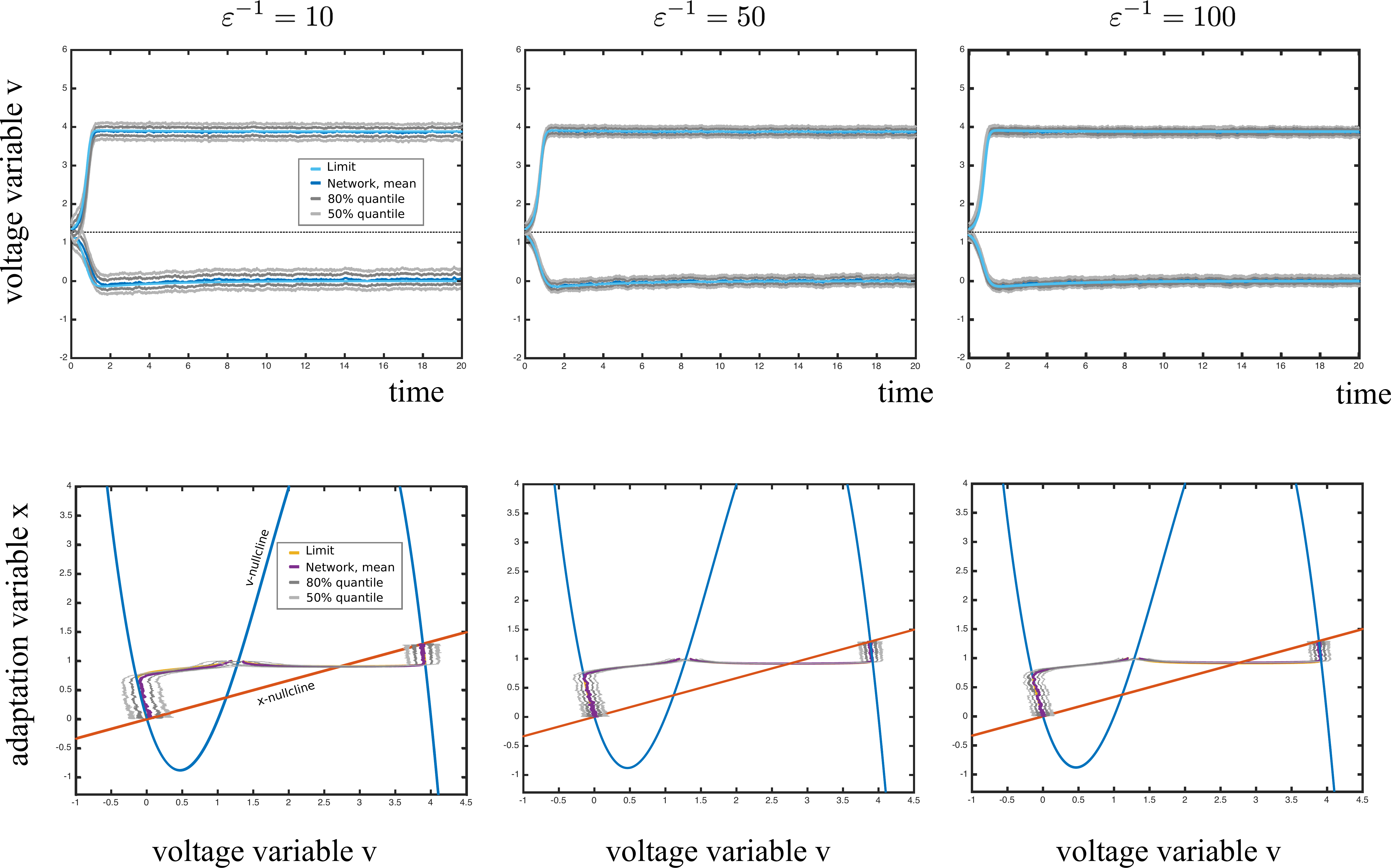}}
\caption{Bistable network: numerical simulations of the network~\eqref{eq:network} with $n=500$ neurons, corresponding to a bistable system (nullclines depicted in phase diagrams, bottom row). Networks parameters: $\sigma=1$ and $\eps^{-1}: 10,\; 50$ and $100$. (top): voltage variable as a function to time, bottom: trajectories in the phase diagram. The two sets of trajectories in each diagram correspond to random initial conditions centered in the attraction bassin of  the two stable fixed points. Curves: theoretical solution, average network variables and 10, 25, 75 and 90\% quantiles (see colorcode in the figure).}
\label{fig:FixedPoints}
\end{figure}

{\bf Multistability } In Fig.~\ref{fig:FixedPoints}, we consider a parameter set for which the limit equation displays two stable solutions (see parameters in the Figure caption). For these parameters, the phase diagram is split into two regions corresponding to the attraction bassins of the two fixed points. We compare the dynamics of the theoretical solution and simulations of a relatively small network $n=500$ with noise of standard deviation $\sigma=1$, and particularly consider the impact of the level of coupling (value of $\eps$) and the centering of the initial condition.

We observe that the network dynamics is closely centered around the theoretical solution; this is visible in the excellent agreement of the theoretical solution and the average voltage and adaption value (averaged over the $n$ neurons) for all three values of $\eps$ tested. This shows that the network also features a clear apparent bistability, as the network stabilizes around the theoretically predicted solutions (because of the finite network size and non-infinite coupling, the network shall randomly switch between the two attractors). Moreover, the bassins of attractions of the network appear to match with the theoretical solution. In particular, in Fig.~\ref{fig:FixedPoints}, we considered initial conditions close from the separatrix ($x(0)=1$, $v(0)=1.2$ or $v(0)=1.35$, on both sides of the value of the voltage of the separatrix at $x(0)=1$, depicted with a dashed line in the top row of Fig.~\ref{fig:FixedPoints}). Moreover, the shrinkage of the distribution around the theoretical solution is illustrated by plotting the trajectories of the 10, 25, 75 and 90 \% quantiles, outlining the voltage and adaptation values associated containing 80\% (between quantiles 10 and 90\%) or 50\% (between quantiles 25\% - 75\%) of the distribution.

{\bf Periodic solutions} In Fig.~\ref{fig:particles2}, we simulate the theoretical limit and the network equation for parameters associated with periodic solutions of the limit. The parameters are identical to those of Fig.~\ref{fig:FixedPoints}, with a decreased the adaptation parameter $a$ and increased external input $I_{ext}$. As in Fig.~\ref{fig:FixedPoints}, we depict the voltage trajectories as a function of time (top row) and the trajectories in the phase plane for the theoretical solution, the mean of the network variables as well as the 10\%, 25\%, 75\% and 90\% quantiles. 
\begin{figure}[h]
\centering
\centerline{\includegraphics[width=.9\textwidth]{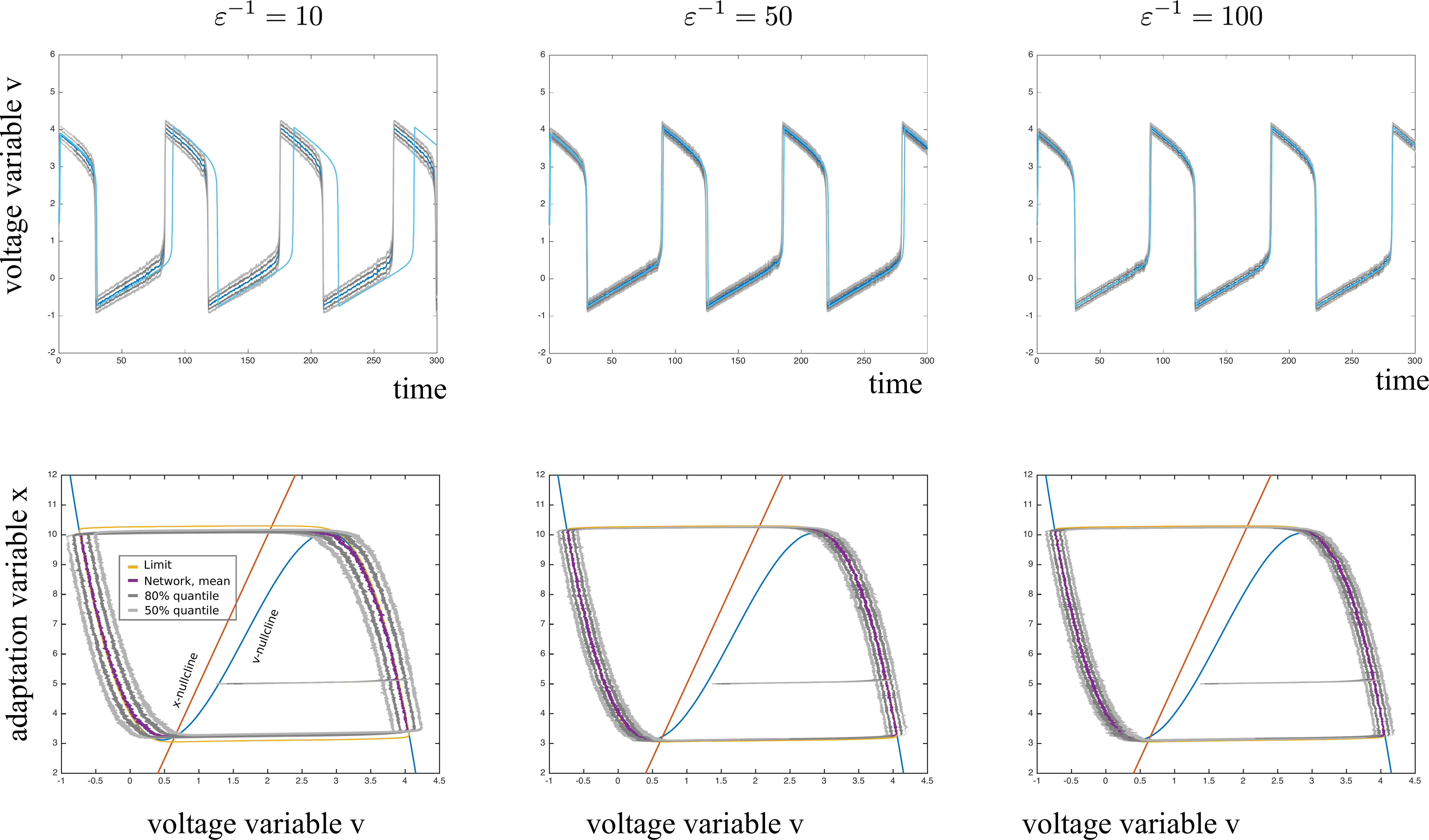}}
\caption{Simulation of the network equation~\eqref{eq:network} in the oscillatory regime. Same color code and parameters as used for Fig.~\ref{fig:FixedPoints}, except for $a=0.03$ and $I_{ext}=4$. }
\label{fig:particles2}
\end{figure}

We observe that, despite the relatively small size of the network, the trajectories are accurately predicted by the theoretical solution, and neurons oscillate periodically and in phase. For $\eps=0.1$, apparently periodic solutions emerge (although, because of the finite network size, these solutions will not be rigorously periodic). Compared to the theoretical solution, \cris{\sout{the }}we notice that network oscillations are slightly faster, and, in the phase space, do not exactly match the theoretical limit near the folds. Indeed, for not small enough $\eps$, the concentration of trajectories is not sufficient to follow the relaxation cycle up to the folds of the voltage nullcline, leading to a faster relaxation cycle and accounting for the distinction between network and theoretical cycles. For $\eps=0.02$ or $0.01$, these differences progressively vanish, and, for $\eps=0.01$, the simulation of the network shows a very good match with the theoretical solution, both in time and in the phase space. 

{\bf Perturbations near the transitions}
The above simulations show that for parameters away from the bifurcations of the deterministic system~\eqref{eq:LimitFhN}, the network equation, even with a relatively small number of neurons, closely follows the dynamics predicted theoretically by their limit $n\to \infty$ and $\eps\to 0$. However, near bifurcations, the finiteness of the number of neurons and the fact that $\eps>0$ can have dramatic effects on the solutions, that are hard to predict. Indeed, because of the interaction term in $\eps^{-1}$, the dynamics can be seen as a slow-fast system, showing a rapid concentration around a Dirac solution at the average value of the system, followed by a slow evolution guided by the flow of~\eqref{eq:LimitFhN}. When this flow brings the system towards a weakly stable regime (for parameters close from a bifurcation point), finite-size effects and fluctuations due to the presence of noise (that are not completely cancelled by an infinitely strong coupling) may lead to deviations from the theoretical limiting behavior, and produce new unpredicted dynamics. Here, we explored numerically these effects by considering parameters in the vicinity of the Hopf instability. 

\cris{Our numerical simulations tend} to show that the behavior of the network deviates from the dynamics predicted by the limiting ordinary differential equation. In Fig.~\ref{fig:transition}, we consider the role of the input current $I_{ext}$. We show in particular that large relaxation oscillations (interpreted as action potentials in neuroscience) appear and stabilize in regimes where the deterministic system describing the limit only shows a stable fixed point. The frequency of the oscillations appears relatively regular, and increases progressively from very slow oscillations to the finite frequency associated to oscillations in the deterministic system as we enter the oscillating regime. 

\begin{figure}[h]
\begin{center}
\subfigure[$I_{0}=5$ - voltage]{\includegraphics[width=.3\textwidth]{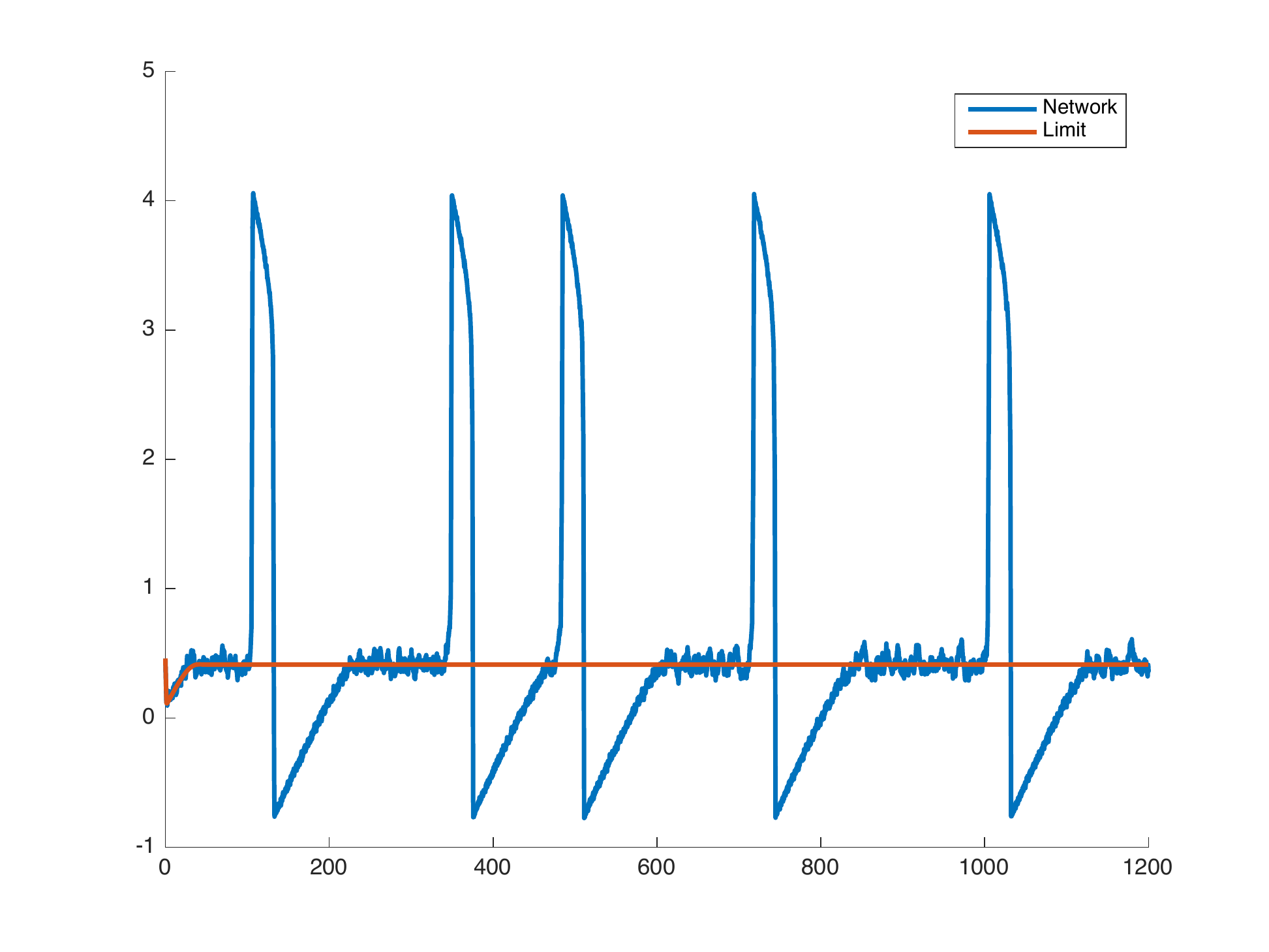}}\quad
\subfigure[$I_{0}=5.4$ - voltage]{\includegraphics[width=.3\textwidth]{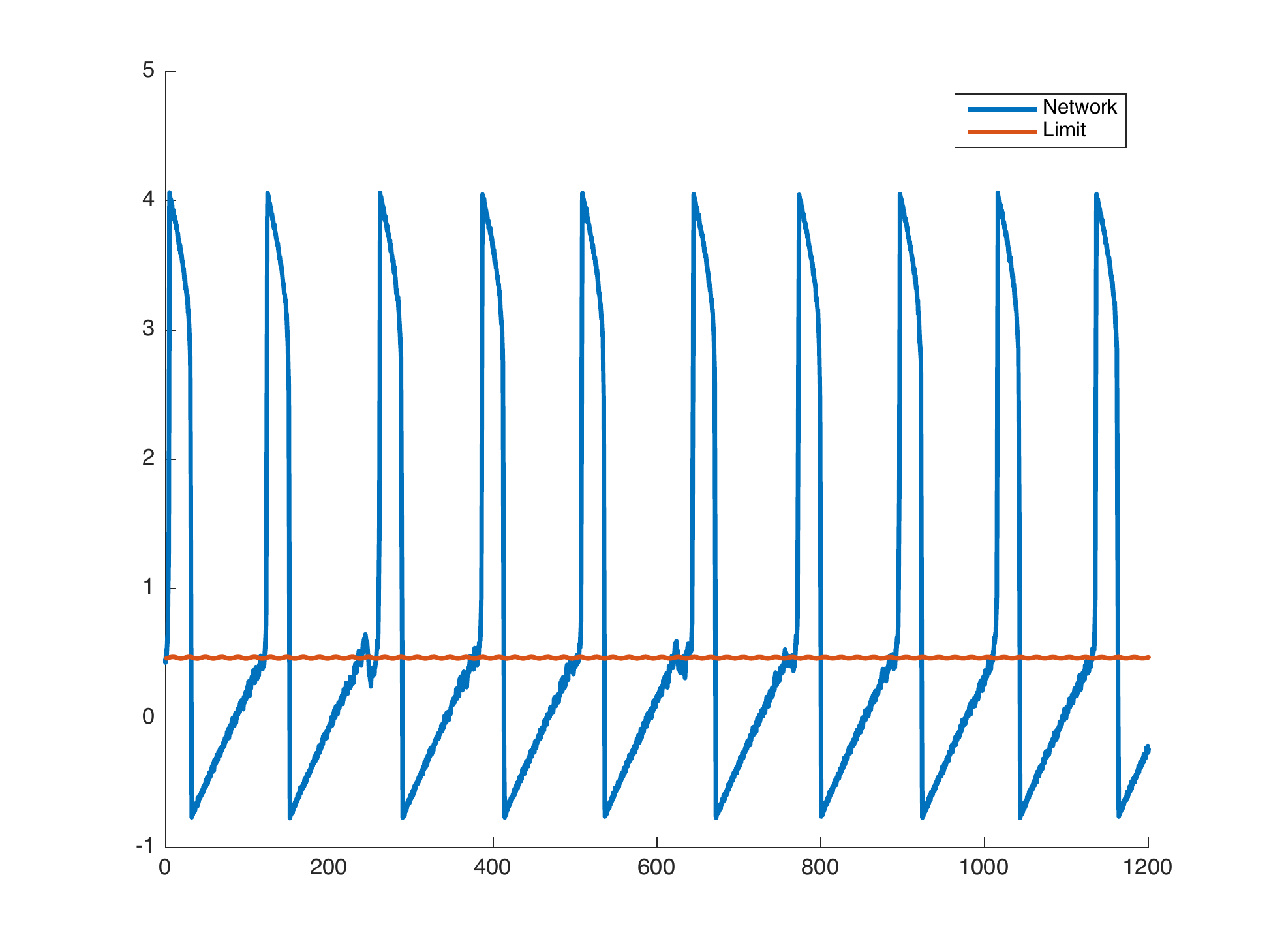}}\quad
\subfigure[$I_{0}=5.7$ - voltage]{\includegraphics[width=.3\textwidth]{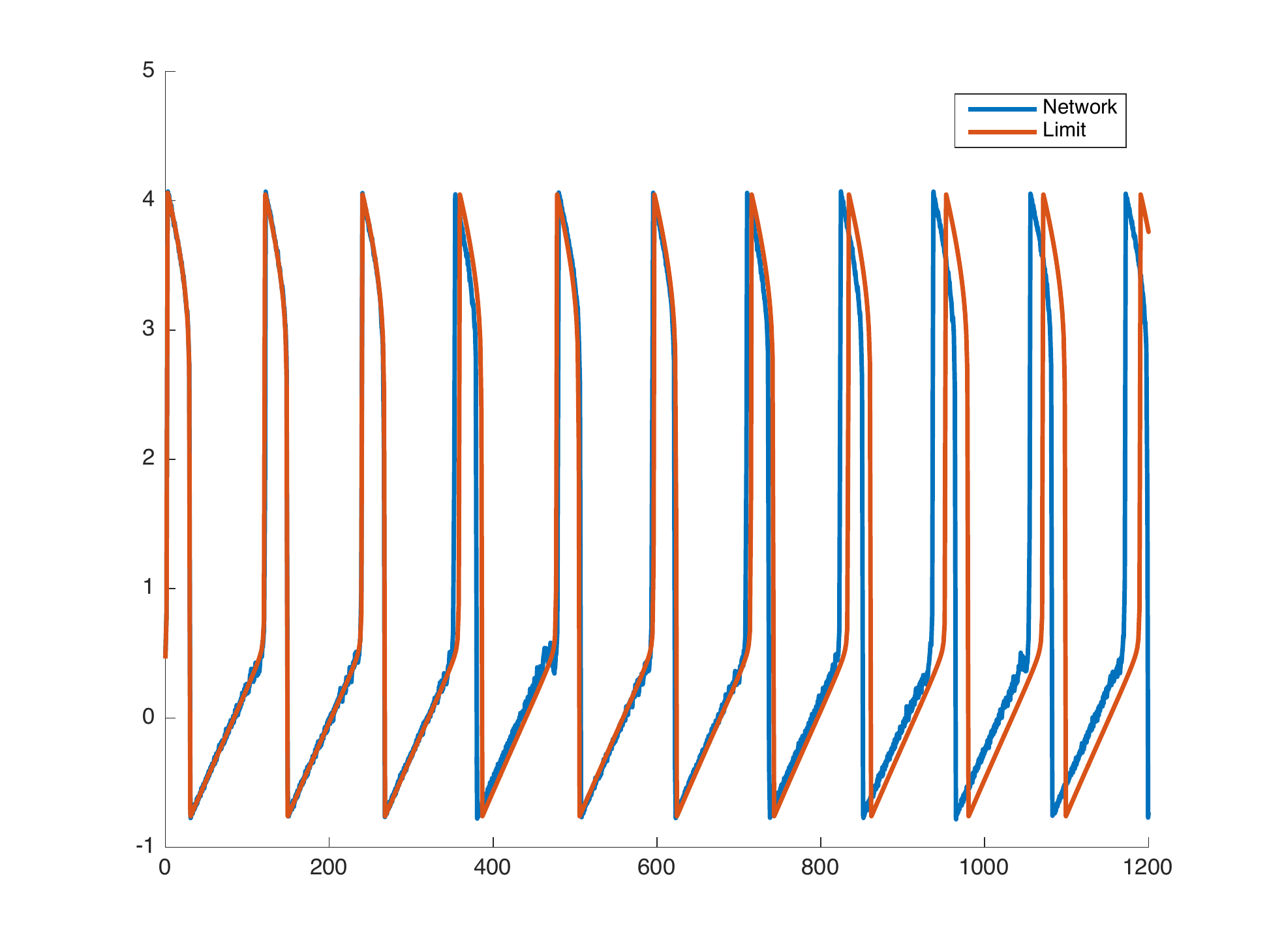}}
\end{center}
\caption{Simulations of the network equations~\eqref{eq:network} through the Hopf bifurcation, with $n=500$, $\eps=0.01$, $\lambda=4$, $a=0.01$ and $b=0.1$, and various values of the input $I_{0}$. Relaxation cycles emerge before the limit equation shows any oscillation, first arising irregularly, and progressively become more regular and locking to the eventually appearing relaxation cycle of the deterministic system.}
\label{fig:transition}
\end{figure}

When the number of neurons is increased and $\eps$ decreased, the range of input values $I_{0}$ for which the behavior of the network deviates from the deterministic limit shrinks, but one can still find, in the vicinity of bifurcations, regimes that significantly deviate from the theoretical limit. In particular, in the regime where the system displays a single attractive fixed point but for parameters close from the Hopf bifurcation, the stochastic network shows quite smooth patterns of activity involving small oscillations around the stable fixed point of the limit, interspersed by large oscillations (spikes). This alternation of small and large oscillations is evocative of mixed-mode oscillations~\cite{desroches2012mixed} observed in particular in multiple timescales dynamics, except that the patterns generated by the network seem irregular in the sense that distinct realizations yields variable alternations of small and large oscillations, as we show in Fig.~\ref{fig:transition2} in a large network with $N=5\,000$ neurons in the vicinity of the Hopf bifurcation. The mathematical study of the behavior of the network in these regimes is not in the scope of this paper. 

\begin{figure}
\begin{center}
\subfigure[Realization 1, before spikes]{\includegraphics[width=.3\textwidth]{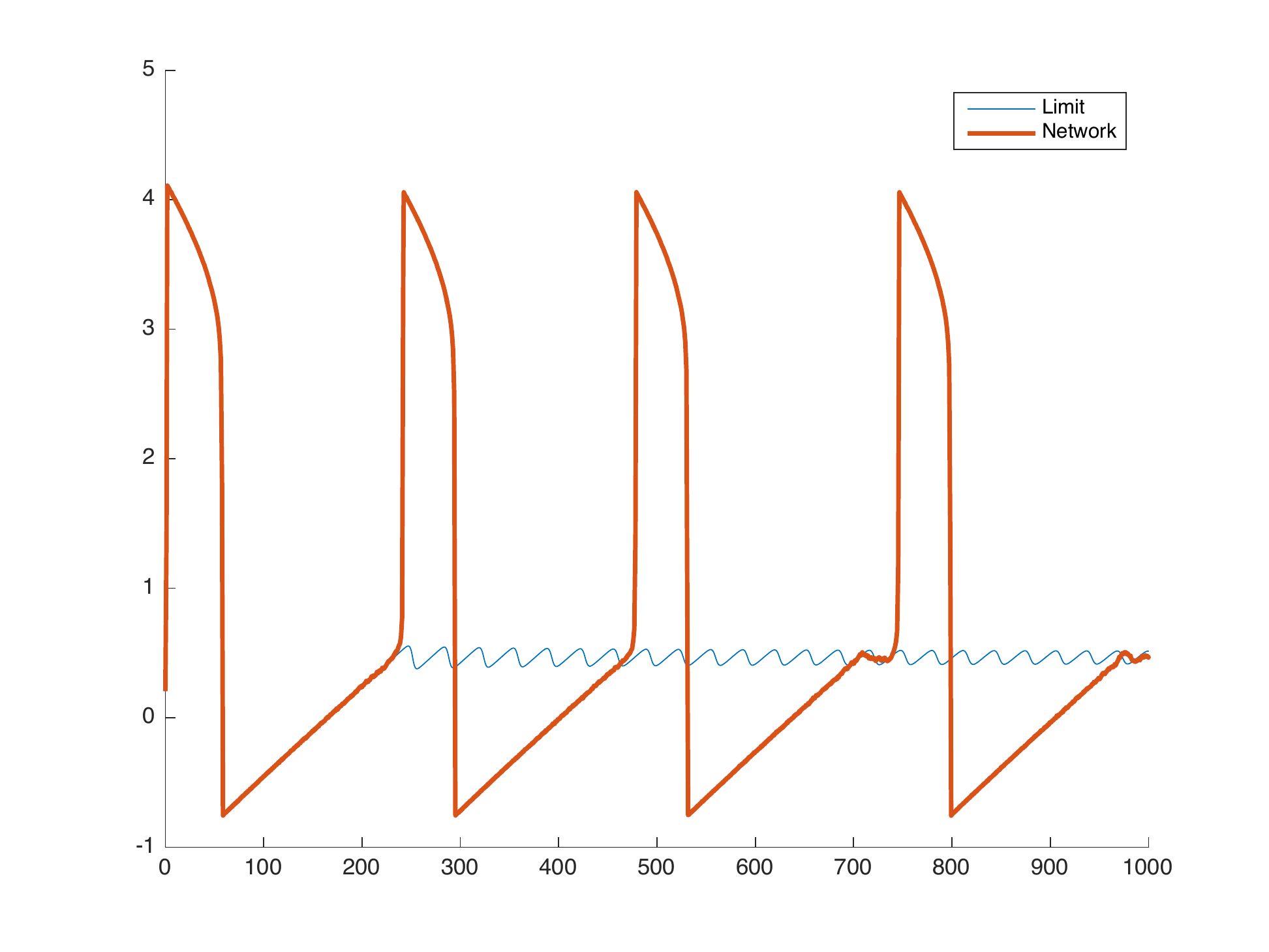}}\quad
\subfigure[Realization 2, before spikes]{\includegraphics[width=.3\textwidth]{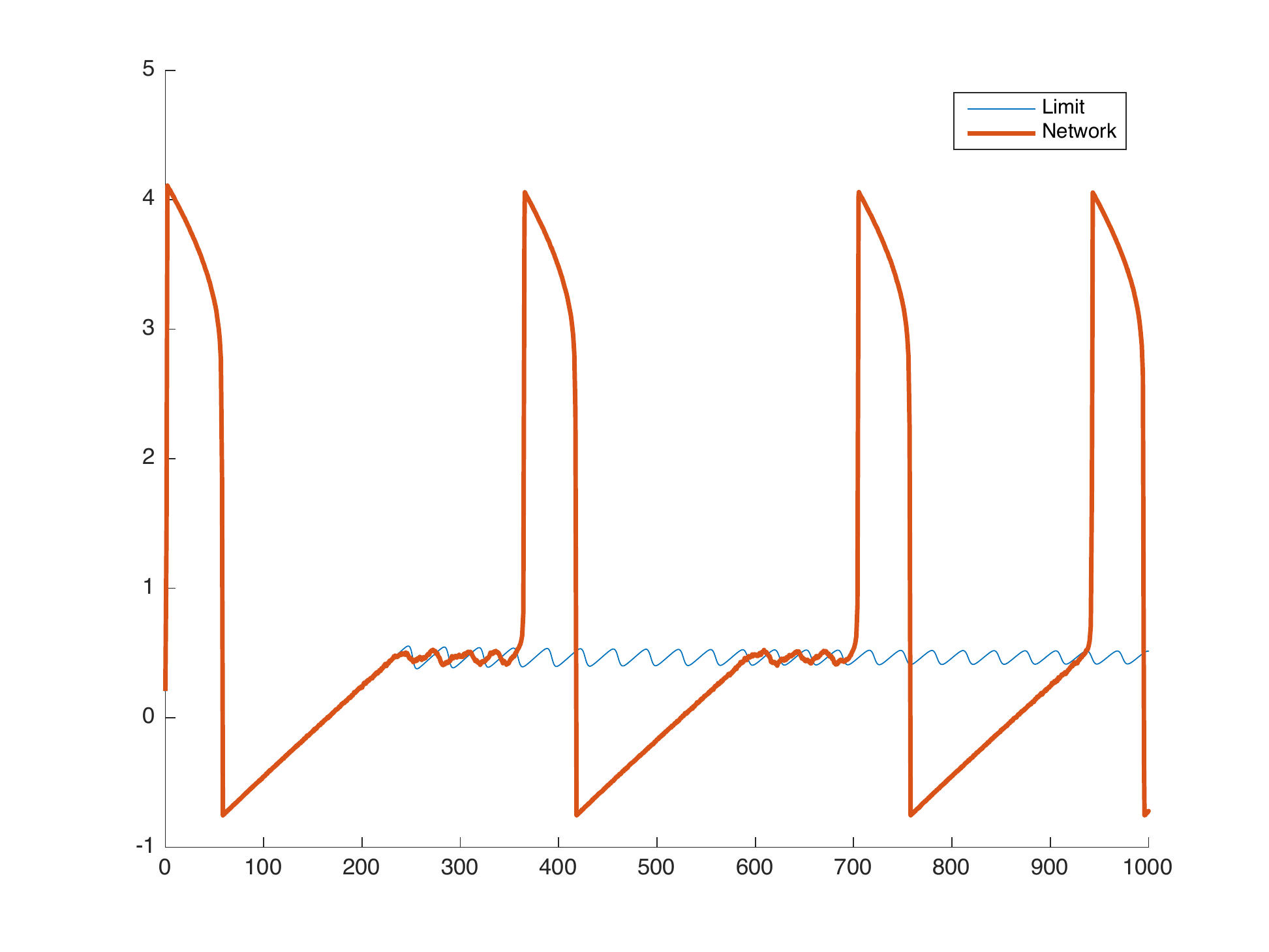}}\quad
\subfigure[Realization 3, before spikes]{\includegraphics[width=.3\textwidth]{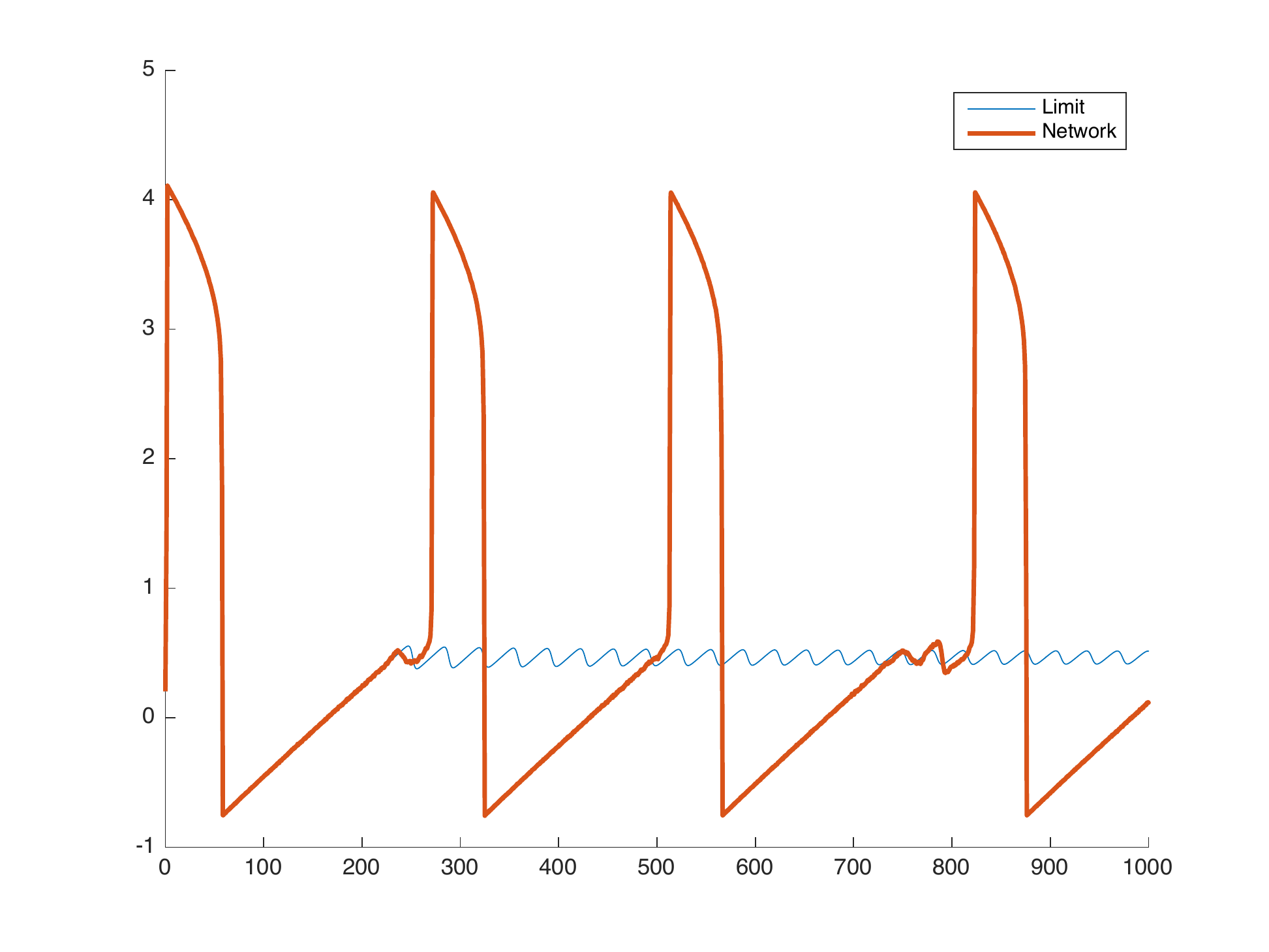}}\\
\subfigure[Realization 1, after spikes]{\includegraphics[width=.3\textwidth]{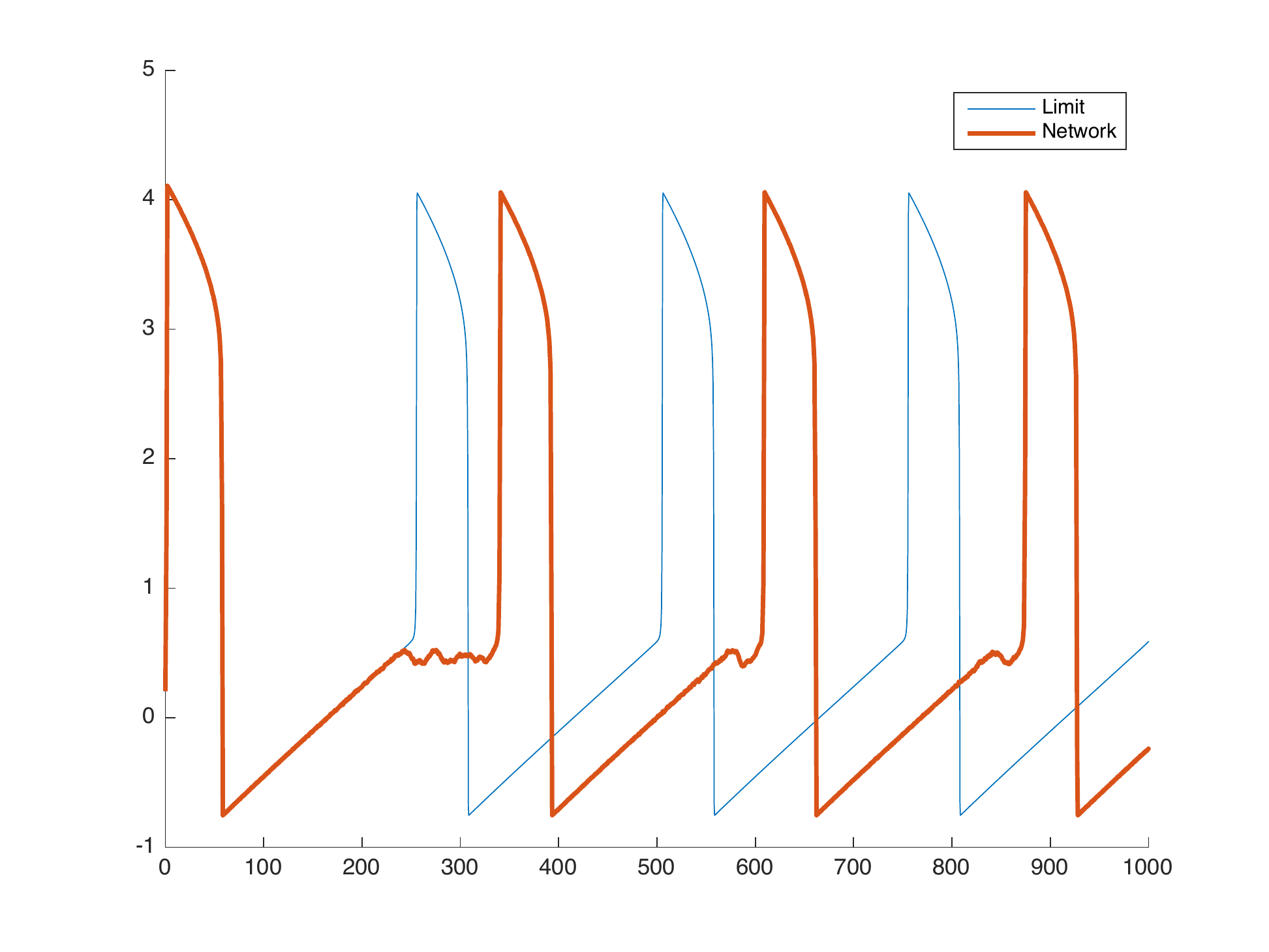}}\quad
\subfigure[Realization 2, after spikes]{\includegraphics[width=.3\textwidth]{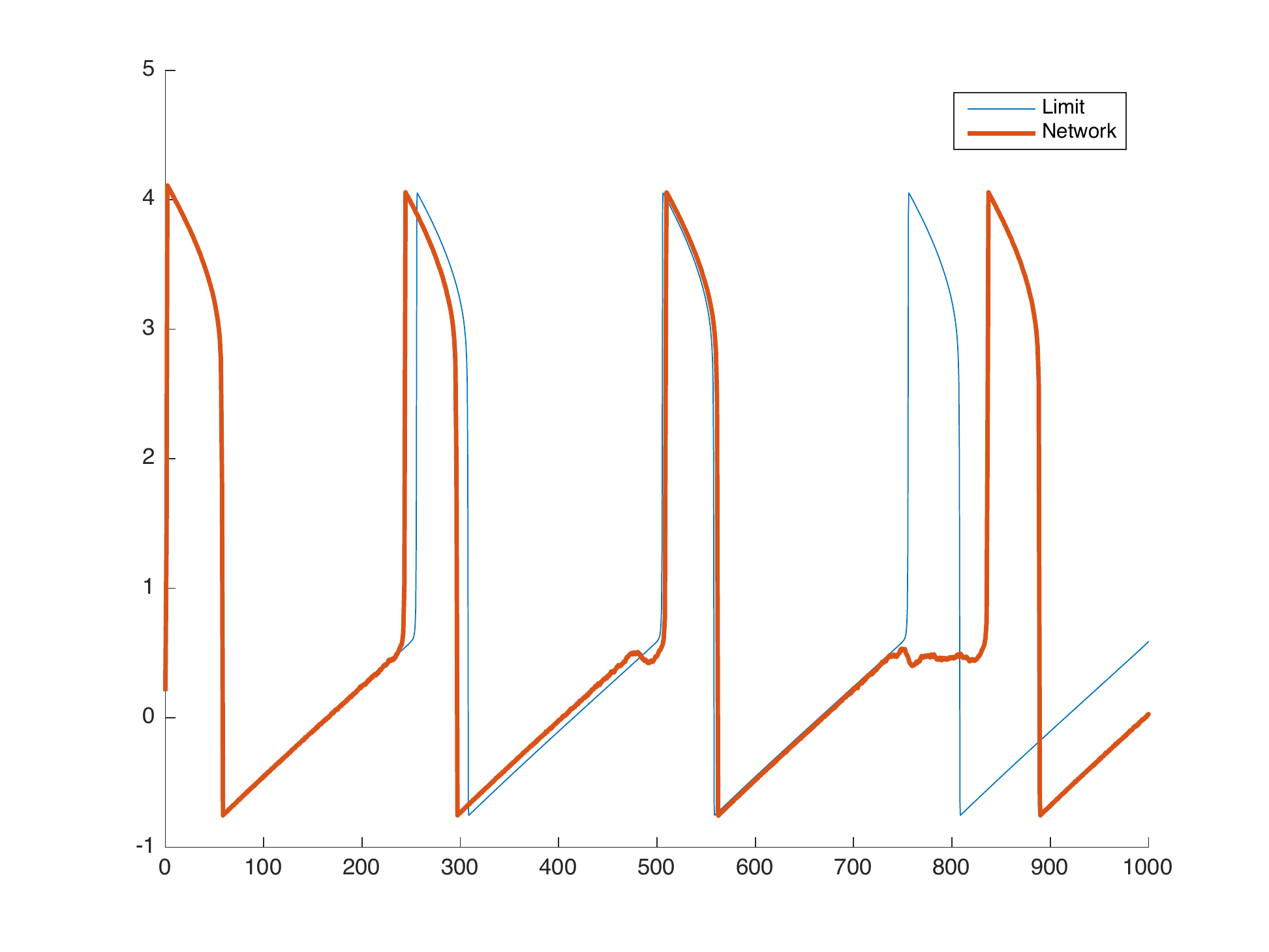}}\quad
\subfigure[Realization 3, after spikes]{\includegraphics[width=.3\textwidth]{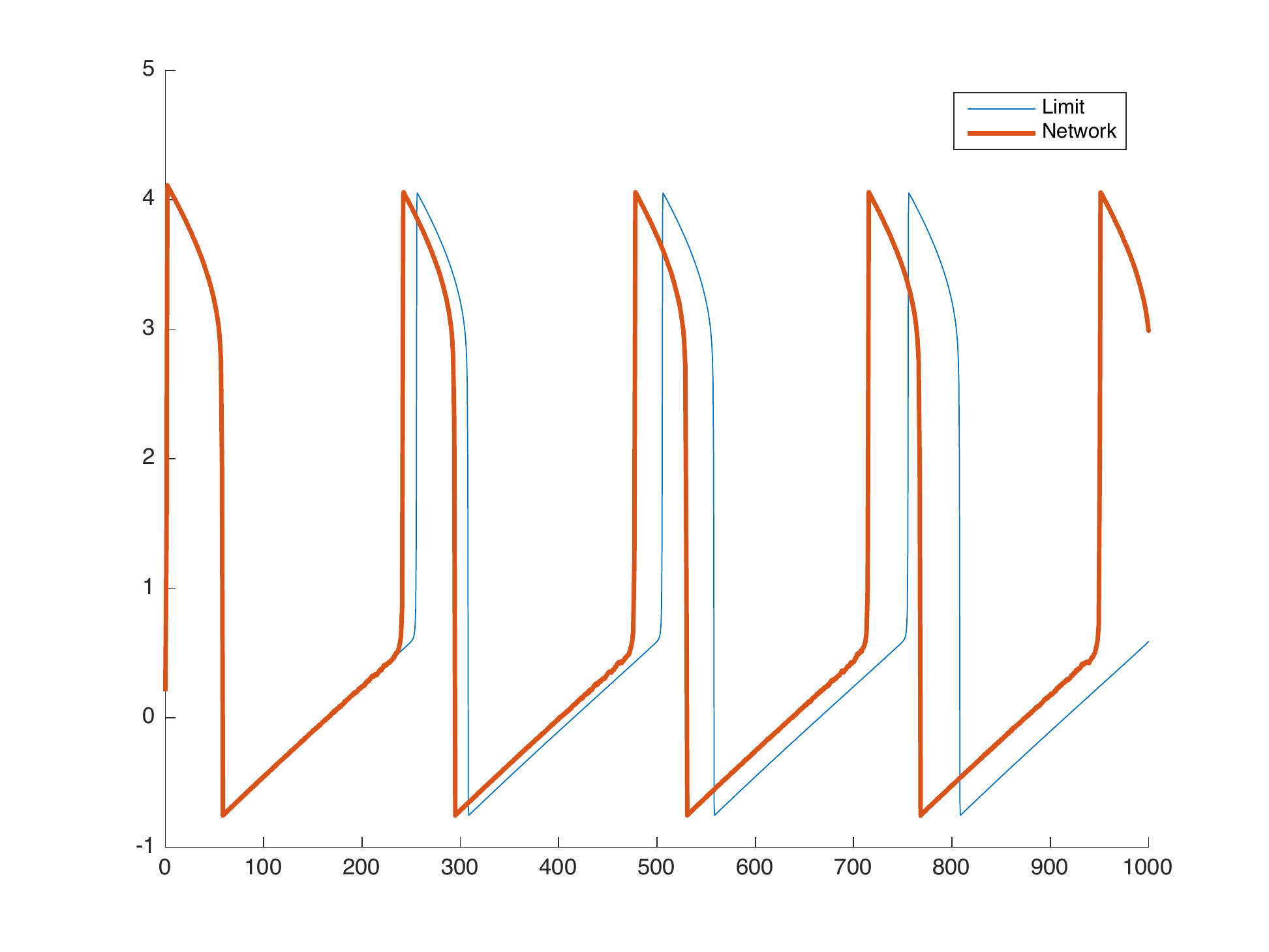}}\\
\end{center}
\caption{Simulations of the network equations~\eqref{eq:network} at the transition from small Hopf oscillations to relaxation oscillations, with $n=5000$, coupling coefficient $\eps^{-1}=220$, $\sigma=0.5$, $\lambda=4$, $a=0.005$ and $b=0.05$, with $I_{0}=5.534$ (top row, before the transition) and $I_{0}=5.5349$ (bottom row, after the transition).}
\label{fig:transition2}
\end{figure}

\section{Discussion}
In this paper, we have pursued the analysis of the mean-field limit of the electrically-coupled FitzHugh-Nagumo system initiated in~\cite{mischler2016kinetic}. That paper focused on the case of small connectivity, and, based on a spectral argument and on the analysis of the uncoupled system, demonstrated that in the limit of vanishing connectivity, there exists a unique stationary solution which is globally attractive. In the present paper, we concentrated on the opposite limit of large connectivity. This limit does not allow us \cris{to compare} the system to the uncoupled linear case, and the nonlinearity of the McKean-Vlasov system becomes prominent. In this limit, the solutions concentrate exponentially around a clamped state in which all neurons have the same voltage, whose magnitude satisfies a simple ordinary differential equation, corresponding to the dynamics of a single neuron. Interestingly, this limit shows complex dynamics and can feature both multiple stable fixed points and periodic orbits. 

In biological terms, the activity regime generally considered correspond to the presence of single fixed point, which can be destabilized in favor a periodic orbit in response to the application of a current. The present theory shows rigorously that large values of electrical coupling (gap junctions) leads to solutions that have periodic laws. In this regime, all neurons oscillate in phase. As reported in the biological literature~\cite{bennett2004electrical,hjorth2009gap}, the present findings supports the idea that gap junctions promote the emergence of synchronization in large-scale networks. Moreover, our developments deal with a system in the presence of large coupling between cells, noise of fixed standard deviation and excitable cells; in that sense, our finding advances the current literature on the topic of junctional synchronization in computational neuroscience, which has essentially focused on cases with small noise or small coupling~\cite{pfeuty2005combined}, or in systems without excitable elements~\cite{ostojic2009synchronization}. 

While the methods used in the present manuscript may be classical in PDEs, to the best of knowledge this paper is the first to introduce such methods in computational neuroscience. The particular nature of the system under consideration required a careful application of existing methods. In particular, the FitzHugh-Nagumo system considered here is characterized by a non globally Lipschitz-continuous drift, and we have developed methods based on truncations of the initial drift. This study has raised several open problems that require deep developments. In particular, the problem we considered here is a double-limit problem: large network size and large connectivity. Here, we have chosen to take the limits in a specific order: first, the large $n$ limit (our starting point, the mean-field FitzHugh-Nagumo equation), and then the large connectivity limit $\eps \to 0$. It remains largely open to understand whether if the two limits commute, and if not, what are the possible dynamical states of the system with large $n$ and small $\eps$. 

In this work, we have shown clamping and synchronization of solutions in the large coupling limit, under relatively mild assumptions on the parameters of the systems. In particular, using a truncation method, we were able to handle non globally-Lipschitz continuous drifts as occurring in the classical FitzHugh-Nagumo equation. However, our results rely on a strong assumption on the concentration of the initial distribution: those were assumed to have tails of leading exponential order $e^{-A(x^{2}+v^{2})/\eps}$. Numerical simulations for fixed initial conditions (not strongly concentrated), we observed two phases in the dynamics: first, a very rapid concentration of the trajectories, and then an onset of the limit dynamics. Heuristically, this rapid concentration may arise because of the dominance of the coupling term in the equation when the voltage variable is away from the network average by an amount of an order of magnitude larger than $\eps$, \johnNew{as outlined formally in the introduction using a change of time, suggesting rapid concentration with a profile $e^{-\gamma t/\eps}$.} Actually, the concentration result shown here may be interpreted as an active process that constrains all voltages to remain within a range of order $\eps$. Understanding these phenomena and short-time behaviors would be an important advance in the understanding of the clamping and synchronization phenomena reported in this paper. 

Besides, we have demonstrated that the limit equation features multiple stable solutions or periodic solutions. An important open problem is to prove that these multiple solutions \cris{persist for }$\eps$ small but non-zero. Our uniform control ensures that for any finite time interval, the solutions closely follow their limit; however, the large time behavior is still to be determined, and in particular the co-existence of multiple stationary solutions, or of periodic solutions is still open. In the case of multiple stationary solutions, methods based on spectral gaps of linearized systems as used in~\cite{mischler2016kinetic} are tempting, but their use becomes much more complex because of the prominence of the nonlinear term in the limit. In the case of periodic solutions, the question is complex, and solutions may rely on characterizing invariant hyperbolic manifolds as in the case of active rotators~\cite{giacomin2012transitions}. \john{Recent work~\cite{lucon2018} extended those methods to a similar FhN system with coupling and noise on the adaptation variable, assuming and exploiting a timescales decomposition between voltage dynamics and mean-field interaction. In this work, McKean-Vlasov systems of excitable systems are shown to display oscillatory behaviors induced by noise and interaction, when coupling is a slow dynamics compared to the intrinsic excitable activity, and coupling occurs in both coordinates. These methods could be instrumental in going beyond the large coupling case and delineating in this system the regions of oscillations in parameter regimes where the underlying FitzHugh-Nagumo system in not in the oscillatory regime. Determining in our system and in the regime of parameters considered the existence of periodic solutions and their stability is an important perspective of this work. Extensions to spatially-extended systems also constitute an interesting perspective of this work. Techniques to address these dynamics were developed in Lipschitz-continuous systems~\cite{touboul2014propagation,touboul2014spatially}, and, very recently, new methods from the domain of PDEs were proposed to handle similar systems with concentrated interactions~\cite{faye:2018}.} Combining these methods and models to Hopf-Cole transforms and large coupling limits could provide a way to simplify the system in large coupling regimes and address the presence of spatio-temporal patterns of dynamics. Eventually, we have seen that near bifurcations of the limit system, complex dynamical solutions emerge. The study of the behavior of the system near those bifurcation also raises important and complex theoretical questions still unsolved. 

\bibliographystyle{siam}
% \bibliography{./FhN_Simple-2dfinal}

\end{document}